\documentclass{article}

\usepackage[english]{babel}
\usepackage[T1]{fontenc}
\usepackage[utf8]{inputenc}

\usepackage{amsmath,amssymb,mathrsfs}
\usepackage{amsthm}
\usepackage{dsfont}
\usepackage{graphicx}
\usepackage{tikz}
\usepackage{empheq}

\usepackage{geometry}
\newcommand{\dis}{\displaystyle}
\newcommand{\lb}{\lambda}
\newcommand{\p}{\partial}
\newcommand{\f}{\frac}
\newcommand{\beq}{\begin{equation}}
\newcommand{\eeq}{\end{equation}}
\renewcommand{\leq}{\leqslant}
\renewcommand{\geq}{\geqslant}
\newcommand{\diff}{\, \mathrm{d} }
\newcommand{\R}{\mathbb{R}}
\newcommand{\Z}{\mathbb{Z}}

\newcommand{\Leb}{\mathrm{L}}

\newcommand{\ddt}{\frac{\p}{\p t}}
\newcommand{\ddx}{\frac{\p}{\p x}}

\newcommand{\dda}{\frac{\p}{\p a}}
\newcommand{\supp}{\mathrm{supp \,}}
\newcommand{\dmz}{\diff \mu (z)}
\newcommand{\e}{\mathrm e}

\DeclareRobustCommand\iff{\;\Longleftrightarrow\;}

\newtheorem{theorem}{Theorem}
\newtheorem{prop}[theorem]{Proposition}
\newtheorem{lem}[theorem]{Lemma}

\title{Steady distribution of the incremental model for bacteria proliferation}

\author{Pierre Gabriel \thanks{Laboratoire de Math\'ematiques de Versailles, UVSQ, CNRS, Universit\'e Paris-Saclay,  45 Avenue des \'Etats-Unis, 78035 Versailles cedex, France. Email: pierre.gabriel@uvsq.fr}
\and Hugo Martin \thanks{Laboratoire Jacques-Louis Lions, CNRS UMR 7598, Sorbonne universit\'e, 4 place Jussieu, 75005 Paris, France. Corresponding author, Email: martin@ljll.math.upmc.fr}}

\begin{document}

\maketitle

\abstract{We study the mathematical properties of a model of cell division structured by two variables -- the size and the size increment -- in the case of a linear growth rate and a self-similar fragmentation kernel. We first show that one can construct a solution to the related two dimensional eigenproblem associated to the eigenvalue $1$ from a solution of a certain one dimensional fixed point problem. Then we prove the existence and uniqueness of this fixed point in the appropriate $\Leb^1$ weighted space under general hypotheses on the division rate. Knowing such an eigenfunction proves useful as a first step in studying the long time asymptotic behaviour of the Cauchy problem.}

\


\

\noindent{\bf AMS Class. No.} Primary: 35Q92, 35P05, 45K05, 45P05, 92D25; Secondary: 35A22, 35B40, 35B65

\
% Timer—structured by age—and sizer—structured by size—models require the bacteria to actively monitor the corresponding parameter during their lives, which is not biologicaly satisfactory. However, r
\section*{Introduction}

In structured population dynamics, finding the structuring variable(s) which best describes a phenomenon is a crucial question. For a population of proliferating cells or bacteria the variables usually considered are age, size (see \cite{Webb,MetzDiekmann,PerthameTransport}) or a combination of both (see \cite{BellAnderson, SinkoStreifer, Hall1991} for modeling and \cite{Webb85,MetzDiekmann,Hall1991,Doumic2007} for mathematical analysis). Recent experimental work highlighted the limits of these models to describe bacteria, and a new variable to trigger division emerged: the \emph{size-increment}, namely the size gained since the birth of the cell (see \cite{Sauls} and references therein for a review of the genesis of the related model). This so called ‘adder principle’ ensures homeostasis with no feedback from the bacteria and explains many experimental data. In this model, bacteria are described by two parameters: their size-increment and their size, respectively denoted by $a$ and $x$ in the following (the choice of letter $a$ is reminiscent from the age variable, since as for the age, the size increment is reset to zero after division). This choice of variables is motivated by the main assumption of the model, which is that the control of the cellular reproduction is provided by the division rate $B$ which is supposed to depend only on $a,$ and the growth rate $g$ which is assumed to depend only on $x.$ With the variables we introduced, the model formulated in \cite{T-A} reads
\[
\left\{
\begin{array}{l l}
\p_t n(t,a,x)+\p_a (g(x)n(t,a,x))+\p_x (g(x)n(t,a,x))+B(a)g(x)n(t,a,x)=0,\quad t\geq0,\ x>a>0,
\vspace{2mm}\\
\dis g(x)n(t,0,x)=4g(2x)\int_0^\infty B(a)n(t,a,2x)\diff a,\qquad t\geq0,\ x>0.
\end{array}
\right.
\]
The function $n(t,a,x)$ represents the number of cells at time $t$ of size $x$ that have grown of an increment $a$ since their birth. The boundary term denotes an equal mitosis, meaning that after division, a mother cell gives birth to two daughters of equal size. However, if this special case of equal mitosis is appropriate to describe the division of some bacterium (\emph{e.g.} E. Coli), it is inadequate for asymetric division (like yeast for instance) or for a fragmentation involving more than two daughters (as in the original model formulated for plant growth in \cite{Hall1991}). In the current paper, we propose to consider more general division kernels. We assume that when a cell of size $x$ divides, it gives birth to a daughter of size $zx$ with a certain probability which depends on $z\in (0,1)$ but is independent of $x.$ Such fragmentation process is usually called self-similar. More precisely the number of daughters with a size between $zx$ and $(z+\!\diff z)x$ is given by $\mu([z,z+\!\diff z]),$ where $\mu$ is a positive measure on $[0,1].$ The model we consider is then formulated as
\begin{subequations}\label{eq:incrementalfrag}
  \begin{empheq}[left=\empheqlbrace]{align}
&\p_t n(t,a,x)+\p_a (g(x)n(t,a,x))+\p_x (g(x)n(t,a,x))+B(a)g(x)n(t,a,x)=0,\quad t\geq0,\ x>a>0, \label{equationnfrag}
\vspace{2mm}\\
&\dis g(x)n(t,0,x)=\int_0^1 g(\frac{x}{z})\int_0^\infty B(a)n(t,a,\frac{x}{z})\diff a\ \frac{\dmz}{z},\qquad t\geq0,\ x>0. \label{bordnaissancenfrag}
\end{empheq}
\end{subequations}
It appears that this model is a particular case of the one proposed in the pioneer work \cite{Hall1991} for plants growing in a single dimension, mixing age and size control. Indeed, in this paper the authors noticed that in the case of a deterministic and positive growth rate, a size/age model is equivalent to a size/birth-size through the relation $a=x-s,$ where $s$ denotes the birth-size (see Figure~\ref{fig:sxy}). They preferred working with the size/birth-size description since in this framework the transport term acts only in the $x$ direction.
In the case when $g$ is independent of $x$ and $B$ is bounded from above and below by positive constants,
it is proved in~\cite{Webb85} for $\mu$ a uniform measure on $[0,1],$ and in~\cite[Chapter V]{MetzDiekmann} for the equal mitosis, that the solutions to Equation~\eqref{eq:incrementalfrag} converge to a stable distribution as time goes to infinity.
In the present paper we propose to study the model~\eqref{eq:incrementalfrag} in the case of a linear growth rate (see \cite{BellAnderson} for a discussion on this hypothesis).
More precisely we are interested in populations which evolve with a stable size and size-increment distribution, {\it i.e.} solutions of the form $n(t,a,x)=h(t)N(a,x).$
The existence of such separable solutions when $g$ is linear was already the topic of~\cite{Hall1991}, but their proof required the equation to be set on a bounded domain and they had to impose {\it a priori} the existence of a maximal size for the population.
In our case no maximal size is prescribed and it brings additional difficulties due to a lack of compactness. To address this problem, we will make the following assumptions.
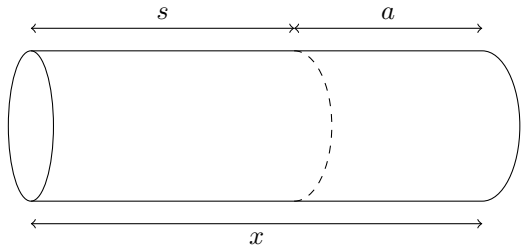
\begin{figure} 
\centering

\begin{tikzpicture}%[scale=1.5]
\draw (0,0) ellipse (.3 and 1);
\draw (0,-1) -- (6,-1)  arc (270:450:0.5 and 1) -- (6,1) -- (0,1);
\draw[dashed] (3.5,-1) arc (270:450:0.5 and 1) (3.5,1);

\draw [<->] (0,-1.3) -- (6,-1.3) node[midway,below] {$x$};
\draw [<->] (0,1.3) -- (3.5,1.3) node[midway, above] {$s$};
\draw [<->] (3.5,1.3) -- (6,1.3) node[midway, above] {$a$};

\end{tikzpicture}
\caption{schematic representation of the variables on an \emph{E. coli} bacterium.\label{fig:sxy}}
\end{figure}

\
First, we want the sum of the daughters’ sizes to be equal to the size of the mother. This rule, called mass conservation, prescribes
\beq \label{massconservation}
\int_0^1 z\dmz=1.
\eeq
We also assume that the division does not produce any arbitrarily small daughter by imposing that the support of $\mu$ is a compact subset of $(0,1),$ which ensures that
\beq \label{probameasure}
\theta:=\inf\supp\mu>0\qquad\text{and}\qquad\exists\eta\in(\theta,1),\ \supp \mu\subset[\theta,\eta].
\eeq
In particular, these assumptions imply that the mean number of daughters $\mu([0,1])$ is finite.
The division rate $B$ is assumed to be a nonnegative and locally integrable function on $\R_+$ such that
\beq \label{supportB}
\exists\, b\geq 0,\quad \supp B = [b,\infty),
\eeq
see~\cite{DG} for instance.
It will be useful in our study to define the associated \emph{survivor function} $\Psi$ by \[
\Psi(a)=\e^{-\int_0^a B(z)\,\diff z}.
\]
For a given increment $a$, $\Psi(a)$ represents the probability that a cell did not divide before having grown at least of $a$ since its birth.
We assume that the function $B$ is chosen in such a way that $\Psi$ tends to zero at infinity,
meaning that all the cells divide at some time.
More precisely we make the following quantitative assumption
\beq \label{decayPsi}
\exists\, k_0> 0, \quad \Psi(a) \underset{+\infty}{=}\mathcal{O}(a^{-k_0}).
\eeq
This assumption on the decay at infinity of the survivor function enables a wide variety of division rates. For instance, it is satisfied if there exists $A>0$ such that
\[\forall\, a \geq A,\quad B(a)\geq \frac{k_0}{a}.\]
The function $B$ being locally integrable, the function $\Psi$ belongs to $W_{loc}^{1,1}(\R_+)$ and \eqref{decayPsi} ensures that its derivative belongs to $\Leb^1(\R_+).$ We can introduce the useful function $\Phi$ defined by
\beq \label{Phi}
\Phi=B\Psi=-\Psi'
\eeq
which is the probability distribution that a cell divides at increment $a.$
Recall that, as in \cite{Hall1991}, we consider the special case of a linear growth rate, namely $g(x)=x.$
In this case, multiplying by the size $x$ and integrating, we obtain
$\f{\!\diff}{\!\diff t} \iint xn(t,a,x)\diff a \diff x = \iint xn(t,a,x)\diff a \diff x,$
and so
\beq \label{conservationlaw}
\iint xn(t,a,x)\diff a \diff x = \e^t\iint xn^0(a,x)\diff a \diff x.
\eeq
This implies that if we look for a solution with separated variables $n(t,a,x)=h(t)N(a,x)$, necessarily $h(t)=h(0)\e^t$. In other words, the Malthus parameter of the population is $1$. This motivates the Perron problem which consists in finding $N=N(a,x)$ solution to
\begin{subequations}
  \begin{empheq}[left=\empheqlbrace]{align}
&\partial_a(xN(a,x)) + \partial_x(xN(a,x))+(1 +xB(a))N(a,x)=0,\qquad x>a>0, \label{equationNfrag} \vspace{2mm}\\
&\dis N(0,x)=\int_0^1 \int_0^{\infty} B(a)N(a,\frac{x}{z})\diff a\ \frac{\dmz}{z^2}, \qquad x>0, \label{bordnaissanceNfrag} \vspace{2mm}\\
& N(a,x)\geq 0, \qquad x\geq a\geq 0,\label{positiviteNfrag} \vspace{2mm}\\
& \int_0^\infty \int_0^x N(a,x) \diff a \diff x=1. \label{normalisationNfrag}
  \end{empheq}
\end{subequations}
It is convenient to define the set $X := \{(a,x)\in\R^2,\ 0\leq a \leq x \},$ and we are now ready to state the main result of the paper.

\begin{theorem}\label{mainthm}
Let $\mu$ be a positive measure on $[0,1]$ satisfying \eqref{massconservation} and \eqref{probameasure}, and $B$ be a nonnegative and locally integrable function on $\R_+$ satisfying \eqref{supportB} such that the associated survivor function $\Psi$ satisfies \eqref{decayPsi}. Then, there exists a unique solution  $N \in \Leb^1(X,(1+(x-a)^2)\diff a \diff x)$ to the eigenproblem \eqref{equationNfrag}--\eqref{normalisationNfrag}.
This solution is expressed as
\beq \label{formethm}
N : (a,x)\in X \mapsto \f{\Psi(a)}{x^2} f(x-a)\eeq
where $f$ is a nonnegative function which satisfies
\[f\in \Leb^1(\R_+, x^l\diff x)\]
for all $l< k_0$, $k_0$ being the positive number given in hypothesis \eqref{decayPsi}, and
\[% \label{supportf}
\supp f = [b_\theta,\infty)
\]
with $b_\theta=\frac{\theta}{1-\theta}b,$ where $\theta$ and $b$ are defined in \eqref{probameasure} and \eqref{supportB} respectively.
\end{theorem}
The fast decay of the function $f$ near zero is a consequence of the form of the support of the fragmentation kernel $\mu.$ Furthermore, this decay is consistent with the decay near zero of the eigenvector for the size equation (see \cite{DG}).
Remark that for any nonnegative and appropriately normalized function $f\in \Leb^1(\R_+),$ the expression given in \eqref{formethm} satisfies \eqref{equationNfrag}, \eqref{positiviteNfrag}, and \eqref{normalisationNfrag}. The proof of Theorem~\ref{mainthm} consists in finding the appropriate function $f$ such that~\eqref{bordnaissanceNfrag} is also satisfied.
This function is obtained as the fixed point of a conservative operator, and this allows us to compute it numerically by using the power iteration (see \cite{BurdenDouglas}). We obtain the function on the left on Figure~\ref{fig:simuN}. On the right is the related density $N(a,x).$

Notice also that for the function $N$ given by \eqref{formethm}, the function $s \mapsto N(a+s,x+s)$ is continuous for any $a\leq x$. It corresponds to the trajectories along the characteristics. 
\begin{figure} 
\begin{center}
\includegraphics[scale=0.34]{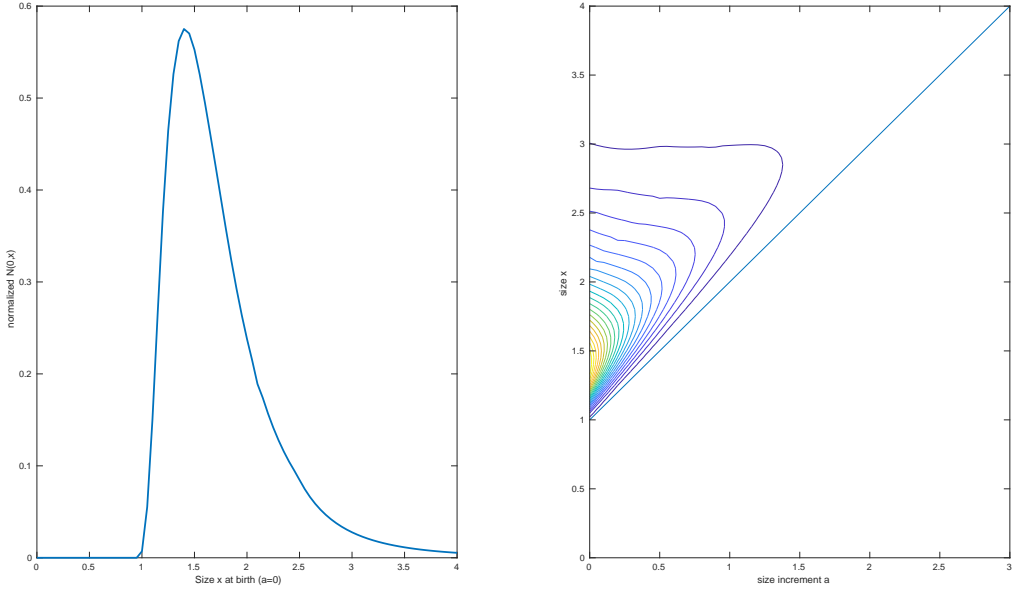}
\caption{Left: simulation of the function $f$ by the power method with $B(a)=\f{2}{1+a}\mathds{1}_{\{1\leq a\}}$ and $\mu(z)=2\delta_{\frac{1}{2}}(z).$ Right: level set of the density $N(a,x)$ obtained from this function $f.$ Straight line: the set $\{x=a+1\}.$ \label{fig:simuN}}
\end{center}
\end{figure}

\medskip

The article is organised as follows. In Section~\ref{sec:reduction} we reduce the Perron eigenvalue problem with two variables to a fixed point problem for an integral operator in dimension one. Section~\ref{sec:existence} is dedicated to proving the existence and uniqueness of the fixed point by using functional analysis and Laplace transform methods. In Section~\ref{sec:entropy} we go through the usefulness of knowing $N$ to develop entropy methods. Finally in Section~\ref{sec:conclusion} we discuss some interesting perspectives.

\section{Transformation into an integral equation}\label{sec:reduction}

Our study consists in constructing a solution to the eigenproblem \eqref{equationNfrag}--\eqref{normalisationNfrag} from the solution of a fixed point problem. First, we notice that the size $x$ of a cell and its size increment $a$ grow at the same speed $g(x)$, so the quantity $x-a$ remains constant: it corresponds to the birth-size of the cell, denoted by $s$. To simplify the equation and obtain horizontal straight lines as characteristics (see Figure \ref{fig:characteristics}), we give a description of the population with size increment $a$ and birth-size $s$, namely we set
\beq \label{relNM}
M(a,s):=N(a,a+s).
\eeq
Thanks to this relation, it is equivalent to prove the existence of an eigenvector for the increment-size system or for the increment/birth-size system. To determine the equation verified by $M$, we compute the partial derivatives of $xN(a,x)=(a+s)M(a,s)$,  which leads to the equation $$\partial_a((a+s)M(a,s))+\left(1 + (a+s)B(a)\right)M(a,s)=0.$$ Writing the non-local boundary condition \eqref{bordnaissanceNfrag} with the new variables takes less calculation and a more interpretation. In \eqref{bordnaissanceNfrag} the number of cells born at size $s$ resulted of the division of cells at size~$\f{s}{z}.$ %of any increment into a fraction of ratio $z$ of itself.
Then the equivalent of \eqref{bordnaissanceNfrag} in the new variables with a linear growth rate is given by \[M(0,s)=\int_\theta^\eta  \int_0^\frac{s}{z} B(a)M(a,\frac{s}{z}-a)\diff a\ \frac{\dmz}{z^2} \] since there is no mass for $a\geq \frac{s}{z}.$ With the relation \eqref{relNM}, it is equivalent to solve \eqref{equationNfrag}--\eqref{normalisationNfrag} and to solve
\begin{subequations}
  \begin{empheq}[left=\empheqlbrace]{align}
&\partial_a((a+s)M(a,s))+(1 +(a+s)B(a))M(a,s)=0,\qquad a,s>0, \label{equationMfrag} \vspace{2mm}\\
&\dis M(0,s)=\int_\theta^\eta  \int_0^\frac{s}{z} B(a)M(a,\frac{s}{z}-a)\diff a\ \frac{\dmz}{z^2},\qquad s>0, \label{bordnaissanceMfrag} \vspace{2mm}\\
& M(a,s)\geq 0, \qquad a,s\geq 0,\label{positiviteMfrag} \vspace{2mm}\\
& \int_{\mathbb{R}_+^2} M(a,s) \diff a \diff s=1. \label{normalisationMfrag}
  \end{empheq}
\end{subequations}
Considering the variable $s$ as a parameter in \eqref{equationMfrag}, we see this equation as an ODE in the variable $a.$ A formal solution is given by
\[
M(a,s)=\frac{\Psi(a)}{(a+s)^2} s^2M(0,s).\]
Having this expression in mind, we note that for any nonnegative function $f\in \Leb^1(\R_+,\diff s)$, the function $M_f$ defined on $\R_+^2$ by
\[M_f : (a,s) \mapsto \frac{\Psi(a)}{(a+s)^2}f(s)\]
is a solution of \eqref{equationMfrag} and satisfies \eqref{positiviteMfrag}. Then it remains to choose the appropriate function $f$ and normalize the related function $M_f$ to solve the whole system \eqref{equationMfrag}--\eqref{normalisationMfrag}. It turns out that this appropriate function $f$ is a fixed point of the operator $T :\Leb^1(\R_+) \to \Leb^1(\R_+)$ defined by
\beq \label{eq:defTfrag}
T f(s) = \int_\theta^{\eta}\int_0^{\f{s}{z}} \Phi(\f{s}{z}-a)f(a)\diff a \dmz,
\eeq
where $\Phi=B\Psi,$ as stated in the following lemma.
\begin{lem} \label{equivfixpt}
The function $M_f$ satisfies \eqref{bordnaissanceMfrag} if and only if $f$ is a fixed point of the operator~$T$.
\end{lem}
\begin{proof}
\begin{align*}
M_f\  \mathrm{~satisfies~} \eqref{bordnaissanceMfrag} & \iff \frac{f(s)}{s^2}=\int_\theta^\eta \int_0^\frac{s}{z} B(a)\frac{\Psi(a)}{(\frac{s}{z})^2} f(\frac{s}{z}-a) \diff a\ \frac{\dmz}{z^2} \\
& \iff f(s)=\int_\theta^\eta \int_0^\frac{s}{z} \Phi(a) f(\frac{s}{z}-a)\diff a\ \dmz \\
& \iff f(s)=\int_\theta^\eta \int_0^\frac{s}{z} \Phi(\frac{s}{z}-a) f(a)\diff a\ \dmz \\
& \iff f(s)=Tf(s)
\end{align*}
\end{proof}

\begin{center}
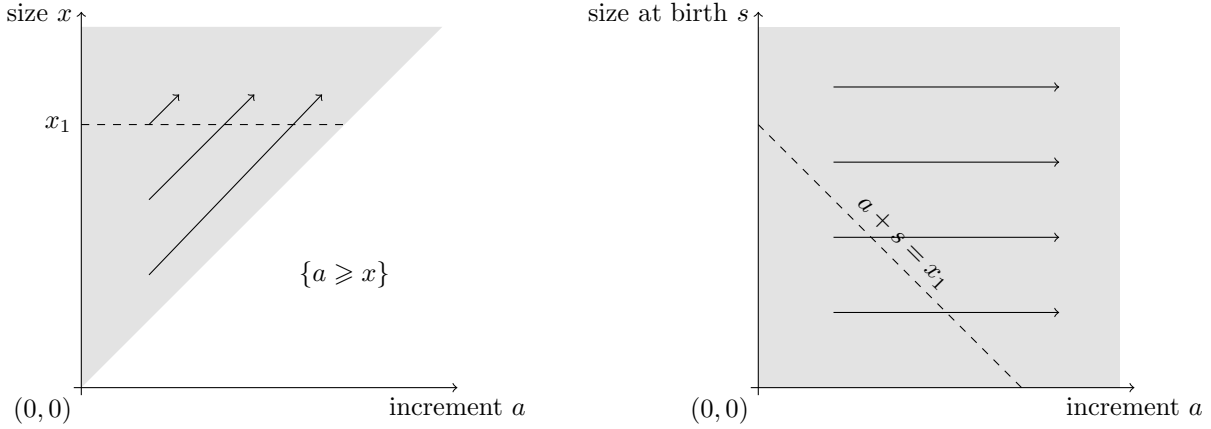
\begin{figure}
\begin{tikzpicture}
\draw (0,0) node[below left] {$(0,0)$};
\draw (0,5) node[left] {$\mathrm{size}~x$};
\draw (5,0) node[below] {$\mathrm{increment}~a$};

\draw (3.5,1.5) node {$\{a\geq x\}$};

\fill [gray!22] (0,0)--(4.8,4.8)--(0,4.8)-- cycle;
\draw [->] (-0.1,0)--(5,0);
\draw [->] (0,-0.1)--(0,5);
%\draw [dashed] (3.5,0)--(3.5,-5);
%\draw (3.5,0) node[above] {$s_{\max}$};

\draw [dashed] (0,3.5)--(3.5,3.5);
\draw (0,3.5) node[left]{$x_1$};

\draw [->] (0.9,1.5)--(3.2,3.9);
\draw [->] (0.9,2.5)--(2.3,3.9);
\draw [->] (0.9,3.5)--(1.3,3.9);

\fill [gray!22] (9,0)--(13.8,0)--(13.8,4.8)--(9,4.8)-- cycle;

\draw [->] (10,1)--(13,1);
\draw [->] (10,2)--(13,2);
\draw [->] (10,3)--(13,3);
\draw [->] (10,4)--(13,4);

\draw [dashed] (12.5,0)--(9,3.5) node[midway, above, sloped] {$a+s=x_1$};

\draw [->] (8.9,0)--(14,0);
\draw [->] (9,-0.1)--(9,5);
\draw (9,0) node[below left] {$(0,0)$};
\draw (14,0) node[below] {$\mathrm{increment}~a$};
\draw (9,5) node[left] {$\mathrm{size~at~birth}~s$};

\end{tikzpicture}
\caption{Domain of the model, with respect to the choice of variables to describe the bacterium. Grey: domain where the bacterias densities may be positive. Arrows: transport.
Left: size increment/size. Right: size increment/birth size. Dashed: location of cells of size $x_1$. \label{fig:characteristics}}

\end{figure}
\end{center}

The operator $T$ can be seen as some kind of \emph{transition operator}: it links the laws of birth size of two successive generations. If $f$ is the law of the parents, then $Tf$ is the law of the birth size of the newborn cells. Indeed, Equation \eqref{eq:defTfrag} can be understood in words as `the number of cells born at size $s$ come from the ones that were born at size $a\in[b_\theta,\frac{s}{z}]$ and elongated of $\f{s}{z}-a$ for all $z\in [\theta,\eta]$ and all $a$ before dividing into new cells’. See \cite{DHKR} for a probabilistic viewpoint on the conservative size equation. %In the following, we extend functions defined on (a subset of) $\R_+$ to $\R$ by $0$ wherever it is not initialy defined. In that perspective, $f$ and $\Phi$ belonging in $\Leb^1(\R)$, we can rewrite $T$ as
%
%%\beq \label{equationconvolutionfrag}
%\[
%Tf(s) = \int_\epsilon^{\eta} \Phi*f(\f{s}{z}) \dmz\]
%%\eeq
%and easily obtain
It is easy to check that $T$ is a continuous linear operator on $\Leb^1(\R_+)$ and that $\|T\|_{\mathcal{L}(\Leb^1(\R_+))}\leq \|\Phi\|_{\Leb^1(\R_+)}=1$ using \eqref{massconservation} and \eqref{Phi}.
The following lemma provides a slightly stronger result.
\begin{lem} \label{stabiliteTL1poids}
For all $l\leq 0$, the operator $T$ maps continuously $\Leb^1(\R_+,s^l\diff s)$ into itself. Additionally, if \eqref{decayPsi} holds true, then $T$ maps continuously $\Leb^1(\R_+,(s^k+s^l)\diff s)$ into itself for any $l\leq0$ and  $k\in[0,k_0).$
\end{lem}

\begin{proof}
We start with $\Leb^1(\R_+,s^l\diff s)$ where $l\leq0.$
For $f\in\Leb^1(\R_+,s^l\diff s)$ and $\beta>\alpha>0$ one has
\begin{align*}
\int_\alpha^\beta |Tf(s)|s^l\diff s & \leq \int_\theta^{\eta}\int_\alpha^\beta s^l\int_0^{\f{s}{z}} \Phi(\f{s}{z}-a) |f(a)|\diff a \diff s \dmz \\
	& \leq \int_\theta^\eta \int_0^{\f{\alpha}{z}} |f(a)| \int_\alpha^\beta \Phi(\f{s}{z}-a) s^l\diff s \diff a \dmz \\
	& \qquad + \int_\theta^\eta \int_{\f{\alpha}{z}}^{\f\beta z} |f(a)| \int_{za}^\beta \Phi(\f{s}{z}-a) s^l\diff s \diff a \dmz \\
	& \leq \int_\theta^\eta \int_0^{\f{\alpha}{z}} |f(a)| \int_{\f{\alpha}{z}-a}^{\f{\beta}{z}-a} \Phi(\sigma) (a+\sigma)^lz^{l+1}\diff \sigma \diff a \dmz  \\
	& \qquad + \int_\theta^\eta \int_{\f{\alpha}{z}}^{\f\beta z} |f(a)| \int_0^{\f{\beta}{z}-a} \Phi(\sigma) (a+\sigma)^lz^{l+1}\diff \sigma \diff a \dmz \\
	& \leq \int_\theta^\eta z^{l+1} \int_0^{\f{\alpha}{z}} |f(a)|a^l\diff a \dmz + \int_\theta^\eta z^{l+1} \int_{\f{\alpha}{z}}^\infty |f(a)|a^l\diff a \dmz \\
	& \leq \theta^l\|f\|_{\Leb^1(\R_+,s^l \diff s)},
\end{align*}
which gives the conclusion by passing to the limits $\alpha\to0$ and $\beta\to+\infty.$

\medskip

For the second part we begin with the proof that under condition \eqref{decayPsi}, for any $k\in[0,k_0)$ one has
\[ %\label{poidsPhi}
\int_0^\infty \Phi(a)a^k\diff a < \infty.
\]
First, recall that $\int_0^\infty \Phi(a)\diff a = 1$ and $\Phi = -\Psi'$. Integrating by parts for $\beta\geq 1$, one has
\[\int_0^\beta \Phi(a)a^k\diff a\leq \int_0^1 \Phi(a)\diff a + \int_1^\beta \Phi(a)a^k\diff a \leq 1 + k\int_1^\beta \Psi(a) a^{k-1}\diff a \]
and the last integral converges when $\beta\to+\infty$ under Assumption~\eqref{decayPsi} because $k<k_0.$
Now let $l\leq0$ and $k\in[0,k_0),$ and let $f\in \Leb^1(\R_+,(s^k+s^l)\diff s)$.
Due to the first part of the proof, we only have to estimate $\int_0^\beta  |Tf(s)|s^k\diff s$ for $\beta>0.$
Since the function $x\mapsto \frac{(1+x)^k}{1+x^k}$ is uniformly bounded on $\R_+,$ there exists of a constant $C>0$ such that $(a+\sigma)^k\leq C(a^k+\sigma^k)$ for all $a,\sigma\geq0,$
and it allows us to write for any $\beta>0$
\begin{align*}
\int_0^\beta  |Tf(s)|s^k\diff s &\leq \int_\theta^{\eta}\int_0^\beta s^k \int_0^{\f{s}{z}} \Phi(\f{s}{z}-a) |f(a)|\diff a  \diff s \dmz \\
   & = \int_\theta^{\eta}\int_0^{\f{\beta}{z}} |f(a)| \int_0^{\f{\beta}{z}-a} \Phi(\sigma) (a+\sigma)^kz^{k+1}\diff \sigma \diff a  \dmz \\
   & \leq C\int_\theta^{\eta} z^{k+1} \int_0^{\f{\beta}{z}} |f(a)|a^k \int_0^{\f{\beta}{z}-a} \Phi(\sigma) \diff \sigma \diff a  \dmz \\
   & \quad + C\int_\theta^{\eta} z^{k+1} \int_0^{\f{\beta}{z}}|f(a)| \int_0^{\f{\beta}{z}-a} \Phi(\sigma) \sigma^k\diff \sigma \diff a \dmz \\
   & \leq C\eta^k\left(\|f\|_{\Leb^1(\R_+,s^k\diff s)}+\|\Phi\|_{\Leb^1(\R_+,s^k\diff s)}\|f\|_{\Leb^1(\R_+)}\right).
\end{align*}

\end{proof}

\section{The fixed point problem}\label{sec:existence}

In this section we prove the existence of a unique nonnegative and normalized fixed point of the operator~$T.$

\medskip

Let us first recall some definitions from the Banach lattices theory (for more details, see \cite{Du, Schaefer}). Let $\Omega$ be a subset of $\R_+$ and $\nu$ be a positive measure on $\Omega.$ The space $\Leb^1(\Omega,\nu)$ is an ordered set with the partial order defined by \[f\geq 0\text{ if and only if } f(s)\geq 0\ \nu\text{-a.e. on }\Omega.\] Furthermore, endowed with its standard norm, the space $\Leb^1(\Omega,\nu)$ is a Banach lattice, \emph{i.e.} a real Banach space endowed with an ordering $\geq$ compatible with the vector structure such that, if $f, g \in \Leb^1(\Omega,\nu)$ and $|f|\geq |g|,$ then $\|f\|_{\Leb^1(\Omega,\nu)} \geq \|g\|_{\Leb^1(\Omega,\nu)}.$ A vector subspace $I\subset\Leb^1(\Omega,\nu)$ is called an ideal if $f\in I, g\in \Leb^1(\Omega,\nu)$ and $|g|\leq |f|$ implies $g\in I.$ For a given operator $A$ defined on $\Leb^1(\Omega,\nu)$, a closed ideal $I$ is $A$-\emph{invariant} if $A(I)\subset I,$ and $A$ is irreducible if the only $A$-invariant ideals are $\{0\}$ and $\Leb^1(\Omega,\nu)$. To each closed ideal $I$ in the Banach lattice $\Leb^1(\Omega,\nu)$ corresponds a subset $\omega \subset \Omega$ such that $I=\{f\in \Leb^1(\Omega,\nu), \supp f \subset \omega\}$. We also define the \emph{positive cone} $\Leb^1_+(\Omega,\nu):=\{f\in \Leb^1(\Omega,\nu)\ |\ f\geq 0\ \nu\text{-a.e. on }\Omega\}$. An operator $A:\Leb^1(\Omega,\nu)\rightarrow \Leb^1(\Omega,\nu)$ is said to be positive if $A(\Leb^1_+(\Omega,\nu))\subset \Leb^1_+(\Omega,\nu).$
\
To prove the existence of an eigenvector associated to the eigenvalue $1$, we will use the following theorem, easily deduced from Krein-Rutman's theorem (see \cite{Du} for instance) and De Pagter's \cite{dePagter1986}.
\begin{theorem} \label{thmKRdP}
Let $A : \Leb^1(\Omega,\nu)\rightarrow \Leb^1(\Omega,\nu)$ be a non-zero positive compact irreducible operator. Then its spectral radius $\rho(A)$ is a nonnegative eigenvalue associated to a nonzero eigenvector belonging to  the positive cone $\Leb^1_+(\Omega,\nu).$
\end{theorem}
%\begin{theorem}[Krein-Rutman]
%Let $X$ be a Banach space, and let $K \subset X$ be a total cone, \emph{i.e.} such that $K-K=\{u-v\ |\ u, v \in K\}$ is dense in $X$. Let $T:X\rightarrow X$ be a non-zero positive compact operator. If $\rho(T)>0$, then $\rho(T)$ is an eigenvalue of $T$ with positive eigenvector.
%\end{theorem}
Due to a lack of compactness of the operator $T$, which is due to the lack of compactness of $\R_+$, we truncate the operator $T$ into a family of operators $(T_\Sigma)_\Sigma$. Let $b_\theta=\f{\theta}{1-\theta}b$ and for $\Sigma>b_\theta$ define the operator $T_\Sigma$ on $\Leb^1((b_\theta,\Sigma))$ by 
\begin{align} 
T_\Sigma f(s) & = \int_\theta^\eta\int_{b_\theta}^{\min (\f{s}{z}, \Sigma)} \Phi(\f{s}{z}-a)f(a)\diff a \dmz \label{TSigma} \\
			& = \nonumber\left\{
\begin{array}{l}
\dis \int_\theta^\eta\int_{b_\theta}^{\f{s}{z}} \Phi(\f{s}{z}-a)f(a)\diff a \dmz,\qquad b_\theta \leq s < \theta \Sigma, \vspace{2mm}\\
\dis \int_\theta^{\f{s}{\Sigma}}\int_{b_\theta}^\Sigma \Phi(\f{s}{z}-a)f(a)\diff a \dmz + \int_{\f{s}{\Sigma}}^{\eta}\int_{b_\theta}^{\f{s}{z}} \Phi(\f{s}{z}-a)f(a)\diff a \dmz,\qquad \theta \Sigma \leq s \leq \eta \Sigma, \vspace{2mm}\\
\dis \int_\theta^\eta\int_{b_\theta}^\Sigma \Phi(\f{s}{z}-a)f(a)\diff a \dmz,\qquad \eta \Sigma < s \leq \Sigma. \end{array} \right.
\end{align}
Defining the lower bound of the domain as $b_\theta$ will ensure the irreducibility of $T_\Sigma.$
We will apply Theorem~\ref{thmKRdP} to the operator $T_\Sigma$ for $\Sigma$ large enough to prove the existence of a pair $(\rho_\Sigma,f_\Sigma)$ such that $T_\Sigma f_\Sigma = \rho_\Sigma f_\Sigma.$ Then, we will prove that there exists a unique $f$ in a suitable space such that $\rho_\Sigma \to 1$ and $f_\Sigma \to f$ as $\Sigma \to \infty,$ with $f$ satisfying $Tf=f.$
\
The following lemma ensures that the truncated operator $T_\Sigma$ is well defined.
\begin{lem}\label{lemmeTTSigma}
If $f \in \Leb^1(\R_+)$ and $\supp f \subset [b_\theta,\Sigma]$, then
\[%\label{TTSigma}
(Tf)_{|[b_\theta,\Sigma]}=T_\Sigma (f_{|[b_\theta,\Sigma]}).
\]
\end{lem}
Lemma \ref{lemmeTTSigma} is a straightforward consequence of the definition of operator $T_\Sigma$ by \eqref{TSigma}. From Lemmas~\ref{stabiliteTL1poids} and~\ref{lemmeTTSigma}, we deduce that $T_\Sigma$ has the same stability mapping properties as $T.$ To prove the compactness of the operator $T_\Sigma$ for a fixed $\Sigma$ and later on that the family $(T_\Sigma)_\Sigma$ is also compact, we use a particular case of a corollary of the Riesz-Fr\'echet-Kolmogorov theorem. First, we define two properties for a bounded subset $\mathcal{F}$ of $\Leb^1(\Omega,\nu)$ with $\Omega$ an open subset of $\R_+,$ and $\nu$ a positive measure, the first on translations and the second on the absence of mass on the boundary of the domain
\begin{align}
&\label{translation}
\left\{\begin{array}{l}
\forall \epsilon >0,\ \forall \omega \subset \subset \Omega,\ \exists \delta \in (0,\mathrm{dist}(\omega,\!^c\Omega)) \mathrm{\ such\ that\ }\\ 
\|\tau_h f- f\|_{\Leb^1(\omega,\nu)}<\epsilon,\ \forall h\in (-\delta,\delta),\ \forall f\in \mathcal{F}
\end{array}\right. \\
&\label{massebord}
\left\{\begin{array}{l}
\forall \epsilon >0,\ \exists \omega \subset \subset \Omega, \mathrm{\ such\ that\ }\\ 
\|f\|_{\Leb^1(\Omega\setminus \omega, \nu)}<\epsilon,\ \forall f\in \mathcal{F}
\end{array}\right.
\end{align}
where $\!^c\Omega$ is understood as the complement of this set in $\R_+.$ \begin{theorem}[from \cite{Brezis}, corollary 4.27] \label{compactBrezis} If $\mathcal{F}$ is a bounded set of $\Leb^1(\Omega,\nu)$ such that \eqref{translation} and \eqref{massebord} hold true, then $\mathcal{F}$ is relatively compact in $\Leb^1(\Omega, \nu).$
\end{theorem}

\subsection{Existence of a principal eigenfunction for $T_\Sigma$}

Using Theorem~\ref{thmKRdP}, we prove the existence of an eigenpair $(\rho_\Sigma, f_\Sigma)$ for the operator $T_\Sigma.$
\begin{prop} \label{existencefSigma}
Let $l$ be a nonpositive number. Under the hypotheses \eqref{massconservation}, \eqref{probameasure} and \eqref{supportB}, there exists a unique normalized eigenvector $f_\Sigma \in \Leb^1_+(\Omega,\nu)$ of the operator $T_\Sigma$ in $\Leb^1((b_\theta,\Sigma),s^l\diff s)$ associated to the spectral radius $\rho_\Sigma$ for every $\Sigma>\max(\frac{1}{1-\theta}b,1)$.
\end{prop}
Applying Theorem~\ref{compactBrezis}, to $\Omega=(b_\theta,\Sigma)$ and the family
\[\mathcal{F}=\{T_\Sigma f, f\in \Leb^1((b_\theta,\Sigma),s^l\diff s), \|f\|_{\Leb^1((b_\theta,\Sigma),s^l\diff s)}\leq 1\},\]
which is bounded in $\Leb^1((b_\theta,\Sigma),s^l\diff s),$ as already shown in the proof of Lemma~\ref{stabiliteTL1poids}, we prove the following Lemma.

\begin{lem} \label{compactTSigma}
Let $l$ be a nonpositive number. Under the hypotheses of Proposition~\ref{existencefSigma}, for all $$\Sigma>\max(\frac{b}{1-\theta},1),$$ the set $\mathcal{F}$ is relatively compact.
\end{lem}

%We will show that $T_\Sigma(\overline{B}_{(\Leb^1((b_\epsilon,\Sigma),s^l\diff s)}(0,1))$ is relatively compact in $\Leb^1((b_\epsilon,\Sigma),s^l\diff s)$. Any compact set in $(b_\epsilon,\Sigma)$ is included in a segment $[\alpha,\beta],$ for $\alpha$ and $\beta$ as close as needed from $b_\epsilon$ and $\Sigma,$ respectively, so to prove lemma \ref{compactTSigma}, we will use a particular case of a corollary from \cite{Brezis}.

%\beq \label{translation}
%\left\{\begin{array}{l}
%\forall \epsilon >0,\ \forall \alpha > b_\gamma,\ \forall \beta < \Sigma,\ \exists \delta \in (0,\min\{\alpha-b_\gamma, \Sigma-\beta\}) \mathrm{\ such\ that\ }\\ 
%\|\tau_h T_\Sigma f-T_\Sigma f\|_{\Leb^1((\alpha, \beta),s^l\diff s)}<\epsilon,\ \forall h\in (-\delta,\delta),\ \forall f\in \overline{B}_{(\Leb^1((b_\epsilon,\Sigma),s^l\diff s)}(0,1)
%\end{array}\right.
%\eeq
%\beq
%\label{massebord}
%\left\{\begin{array}{l}
%\forall \epsilon >0,\ \exists \alpha> b_\gamma,\ \exists \beta < \Sigma, \mathrm{\ such\ that\ }\\ 
% \|f\|_{\Leb^1((b_\gamma,\alpha),s^l\diff s)}+\|f\|_{\Leb^1((\beta,\Sigma),s^l\diff s)}<\epsilon,\ \forall f\in \overline{B}_{(\Leb^1((b_\epsilon,\Sigma),s^l\diff s)}(0,1)
%\end{array}\right.
%\eeq

\begin{proof}[Proof of Lemma~\ref{compactTSigma}]
The set $\mathcal{F}$ is bounded due to the continuity of $T$ proven in Lemma~\ref{stabiliteTL1poids}. First, we show that \eqref{translation} is satisfied. Any compact set in $(b_\theta, \Sigma)$ is included in a segment $[\alpha, \beta].$ Without loss of generality, we take $b_\theta<\alpha<\theta\Sigma,\ \eta\Sigma<\beta<\Sigma.$ It is sufficient to treat the case $h$ positive, so let $0\leq h<\min(\theta\Sigma-\alpha, \Sigma-\beta, \Sigma(\eta-\theta)).$ Since $T_\Sigma f$ is piecewise defined, we have to separate the integral on $[\alpha,\beta]$ into several parts, depending on the interval $s$ and $s+h$ belong to, and we obtain
\begin{align*}
&\int_\alpha^\beta \left| T_\Sigma f(s+h)-T_\Sigma f(s)\right|s^l\diff s \\
&\leq \int_\alpha^{\theta\Sigma-h} \left| T_\Sigma f(s+h)-T_\Sigma f(s)\right| s^l\diff s =:(A)\\
&\quad+\int_{\theta\Sigma-h}^{\theta\Sigma} \left| T_\Sigma f(s+h)-T_\Sigma f(s)\right| s^l\diff s =:(B)\\
&\qquad +\int_{\theta\Sigma}^{\eta\Sigma-h} \left| T_\Sigma f(s+h)-T_\Sigma f(s)\right| s^l\diff s =:(C) \\
&\qquad \quad +\int^{\eta\Sigma}_{\eta\Sigma-h} \left| T_\Sigma f(s+h)-T_\Sigma f(s)\right| s^l\diff s =:(D)\\
&\qquad \qquad +\int_{\eta\Sigma}^\beta \left| T_\Sigma f(s+h)-T_\Sigma f(s)\right| s^l\diff s =:(E).
\end{align*}
since for $(A),\ (C)$ and $(E),$ $T_\Sigma f$ and $\tau_h T_\Sigma f$ have the same expression, the same kind of calculations apply, so we only treat $(C),$ which has the most complicated expression. 

\begin{align*}
(C)&= \int_{\theta\Sigma}^{\eta\Sigma-h} \left| T_\Sigma f(s+h)-T_\Sigma f(s)\right| s^l\diff s \\
& \leq \int_{\theta\Sigma}^{\eta\Sigma-h} \left| \int_\theta^{\f{s+h}{\Sigma}}\int_{b_\theta}^\Sigma \Phi(\f{s+h}{z}-a)f(a)\diff a \dmz - \int_\theta^{\f{s}{\Sigma}}\int_{b_\theta}^\Sigma \Phi(\f{s}{z}-a)f(a)\diff a \dmz \right| s^l\diff s \\
& \quad + \int_{\theta\Sigma}^{\eta\Sigma-h} \left| \int^{\eta}_{\f{s+h}{\Sigma}}\int_{b_\theta}^{\f{s+h}{z}} \Phi(\f{s+h}{z}-a)f(a)\diff a \dmz - \int^\eta_{\f{s}{\Sigma}}\int_{b_\theta}^{\f{s}{z}} \Phi(\f{s}{z}-a)f(a)\diff a \dmz \right| s^l\diff s \\
& \leq \int_{\theta\Sigma}^{\eta\Sigma-h} s^l\int_\theta^{\f{s}{\Sigma}} \int_{b_\theta}^\Sigma \left| \Phi(\f{s+h}{z}-a)-\Phi(\f{s}{z}-a)\right| |f(a)|\diff a \dmz \diff s =:(C1)\\
& \quad + \int_{\theta\Sigma}^{\eta\Sigma-h} s^l\int_{\f{s}{\Sigma}}^{\f{s+h}{\Sigma}} \int_{b_\theta}^\Sigma \Phi(\f{s+h}{z}-a) |f(a)| \diff a \dmz \diff s =:(C2)\\
& \qquad + \int_{\theta\Sigma}^{\eta\Sigma-h} s^l\int^{\eta}_{\f{s+h}{\Sigma}}\int_{b_\theta}^{\f{s}{z}} \left| \Phi(\f{s+h}{z}-a)-\Phi(\f{s}{z}-a)\right||f(a)|\diff a \dmz \diff s =:(C3)\\
& \quad \qquad  + \int_{\theta\Sigma}^{\eta\Sigma-h} s^l\int_{\f{s+h}{\Sigma}}^{\eta} \int_{\f{s}{z}}^{\f{s+h}{z}}\Phi(\f{s+h}{z}-a) |f(a)| \diff a \dmz \diff s =:(C4)\\
& \qquad \qquad + \int_{\theta\Sigma}^{\eta\Sigma-h} s^l\int_{\f{s}{\Sigma}}^{\f{s+h}{\Sigma}} \int_{b_\theta}^{\f{s}{z}}\Phi(\f{s+h}{z}-a) |f(a)| \diff a \dmz \diff s =:(C5)
\end{align*}

The integrals $(C1)$ and $(C3)$  are dealt with in the same way, and we have the following estimate

\begin{align*}
(C1)&=\int_{\theta\Sigma}^{\eta\Sigma-h} s^l\int_\theta^{\f{s}{\Sigma}} \int_{b_\theta}^\Sigma \left| \Phi(\f{s+h}{z}-a)-\Phi(\f{s}{z}-a)\right| |f(a)|\diff a \dmz \diff s  \\
	&=\int_\theta^{\eta-\frac{h}{\Sigma}}\int_{b_\theta}^\Sigma |f(a)|\int_{z\Sigma}^{\eta\Sigma-h} \left| \Phi(\f{s+h}{z}-a)-\Phi(\f{s}{z}-a)\right|s^l\diff s \diff a \dmz  \\
	&=\int_\theta^{\eta-\frac{h}{\Sigma}}z^{l+1}\int_{b_\theta}^\Sigma |f(a)|\int_{\Sigma-a}^{\frac{\eta\Sigma-h}{z}-a}|\tau_{\frac{h}{z}}\Phi(\sigma)-\Phi(\sigma)|(a+\sigma)^l\diff \sigma \diff a \dmz \\
	&=\int_\theta^{\eta-\frac{h}{\Sigma}}z^{l+1}\int_{b_\theta}^\Sigma |f(a)|a^l\int_{\Sigma-a}^{\frac{\eta\Sigma-h}{z}-a}|\tau_{\frac{h}{z}}\Phi(\sigma)-\Phi(\sigma)|\diff \sigma \diff a \dmz \\
	&\leq\theta^l\sup_{\varepsilon \in [\theta,\eta]} \|\tau_{\frac{h}{\varepsilon}}\Phi-\Phi\|_{\Leb^1(\R_+)}.
\end{align*}
These integrals are as small as needed when $h$ is small enough, due to the continuity of the translation in $\Leb^1(\R_+).$ For $(C2)$ one has 
\begin{align*}
(C2)&=\int_{\theta\Sigma}^{\eta\Sigma-h}s^l \int_{\f{s}{\Sigma}}^{\f{s+h}{\Sigma}} \int_{b_\theta}^\Sigma \Phi(\f{s+h}{z}-a) |f(a)| \diff a \dmz \diff s \\
	&=\int_\theta^{\theta+\frac{h}{\Sigma}}\int_{b_\theta}^\Sigma |f(a)| \int_{\theta\Sigma}^{z\Sigma}\Phi(\f{s+h}{z}-a) s^l\diff s \diff a \dmz \\
	&\quad +  \int_{\theta+\frac{h}{\Sigma}}^{\eta-\frac{h}{\Sigma}}\int_{b_\theta}^\Sigma |f(a)| \int_{z\Sigma-h}^{z\Sigma}\Phi(\f{s+h}{z}-a) s^l\diff s \diff a \dmz \\
	&\qquad +\int_{\eta-\frac{h}{\Sigma}}^\eta\int_{b_\theta}^\Sigma |f(a)| \int_{z\Sigma-h}^{\eta\Sigma-h}\Phi(\f{s+h}{z}-a) s^l\diff s\diff a \dmz \\
	&\leq\int_{\theta}^{\eta}\int_{b_\theta}^\Sigma |f(a)| \int_{z\Sigma-h}^{z\Sigma}\Phi(\f{s+h}{z}-a) s^l\diff s \diff a \dmz \\
	&=\int_{\theta}^{\eta}z^{l+1}\int_{b_\theta}^\Sigma |f(a)| \int_{\Sigma-\frac{h}{z}-a}^{\Sigma-a}\Phi(\sigma+\frac{h}{z}) (a+\sigma)^l\diff \sigma \diff a \dmz \\
	&\leq\int_{\theta}^{\eta}z^{l+1}\int_{b_\theta}^\Sigma |f(a)|a^l\int_{\Sigma-\frac{h}{z}-a}^{\Sigma-a}\Phi(\sigma+\frac{h}{z})\diff\sigma \diff a  \dmz \\
	&\leq \theta^l\sup_{|I|=\frac{h}{\theta}}\int_I \Phi(a)\diff a
\end{align*}
which is small when $h$ is small since $\Phi$ is a probability density. To deal with $(C4),$ we use Fubini's theorem and some changes of variables to obtain
\begin{align*}
(C4)&= \int_{\theta\Sigma}^{\eta\Sigma-h}s^l \int_{\f{s+h}{\Sigma}}^{\eta} \int_{\f{s}{z}}^{\f{s+h}{z}}\Phi(\f{s+h}{z}-a) |f(a)| \diff a \dmz \diff s \\
	&=\int_{\theta+\frac{h}{\Sigma}}^\eta\int_{\theta\Sigma}^{z\Sigma-h}s^l\int_{\frac{s}{z}}^{\frac{s+h}{z}} \Phi(\f{s+h}{z}-a) |f(a)| \diff a \diff s \dmz \\
	&=\int_{\theta+\frac{h}{\Sigma}}^\eta\int_{\theta\Sigma}^{z\Sigma-h}s^l\int_{-\frac{h}{z}}^0 \Phi(\f{h}{z}+a') |f(\frac{s}{z}-a')| \diff a' \diff s \dmz \\
	&=\int_{\theta+\frac{h}{\Sigma}}^\eta\int_{-\frac{h}{z}}^0 \Phi(\f{h}{z}+a')\int_{\theta\Sigma}^{z\Sigma-h}  |f(\frac{s}{z}-a')| s^l\diff s\diff a'  \dmz \\
	&=\int_{\theta+\frac{h}{\Sigma}}^\eta z^{l+1}\int_{-\frac{h}{z}}^0 \Phi(\f{h}{z}+a')\int_{\f{\theta\Sigma}{z}-a'}^{\Sigma-\frac{h}{z}-a'}  |f(\sigma)| (\sigma+a')^l\diff \sigma\diff a'  \dmz \\
	&\leq \int_{\theta+\frac{h}{\Sigma}}^\eta z^{l+1}\int_{-\frac{h}{z}}^0\Phi(\f{h}{z}+a')\diff a'  \dmz \\
	&\leq \theta^l\left(1-\Psi(\frac{h}{\theta})\right),
\end{align*}
and the continuity of $\Psi$ at $0$ provides the wanted property. Finally, noticing that $(C5)\leq (C2)$ because the integrand are nonnegative, we obtain the desired control on the integral $(C).$ Now for the integral $(B),$ which is dealt with as would be $(D),$ we write
\begin{align*}
(B)&=\int_{\theta\Sigma-h}^{\theta\Sigma}\left|\int_\theta^{\frac{s+h}{\Sigma}}\int_{b_\theta}^\Sigma\Phi(\frac{s+h}{z}-a)f(a)\diff a \dmz + \int_{\frac{s+h}{\Sigma}}^\eta\int_{b_\theta}^{\frac{s+h}{z}}\Phi(\frac{s+h}{z}-a)f(a)\diff a \dmz \right.\\
&\quad \left.-\int_\theta^\eta \int_{b_\theta}^\frac{s}{z}\Phi(\frac{s}{z}-a)f(a)\diff a \dmz\right|s^l\diff s \\
& \leq \int_\theta^\eta\int_{b_\theta}^{\Sigma}|f(a)|\int_{\theta\Sigma-h}^{\theta\Sigma}\left[\Phi(\frac{s+h}{z}-a)+\Phi(\frac{s}{z}-a)\right]s^l\diff s \diff a \dmz\\
& \leq \int_\theta^\eta z^{l+1}\int_{b_\theta}^{\Sigma}|f(a)|a^l\int_{\frac{\theta\Sigma-h}{z}-a}^{\frac{\theta\Sigma}{z}-a}\left[\Phi(\sigma+\frac{h}{z})+\Phi(\sigma)\right]\diff \sigma \diff a \dmz\\
& \leq 2\,\theta^l\sup_{|I|=\frac{h}{\theta}}\int_I \Phi(a)\diff a
\end{align*}
and again the last term vanishes as $h$ vanishes.
We now show that there is no mass accumulation at the boundary of the domain $(b_\theta,\Sigma),$ \emph{i.e.} that \eqref{massebord} holds true. For $\Sigma > \frac{1}{1-\theta}b,$ we haves $b_\theta < \theta \Sigma$ and we can choose $\alpha < \theta \Sigma,$ so that for all $s\in (b_\theta,\alpha),\ \frac{s}{\Sigma} < \theta.$ With the expression of $T_\Sigma f(s),$ we have
\begin{align}
		& \int_{b_\theta}^\alpha |T_\Sigma f(s)|s^l\diff s \nonumber\\
		& \leq \int_{b_\theta}^\alpha s^l\int^\eta_\theta \int_{b_\theta}^{\frac{s}{z}}\Phi(\f{s}{z}-a) |f(a)|\diff a \dmz  \diff s \nonumber\\
		& \leq \int_\theta^{\eta}\int_{b_\theta}^{\frac{b_\theta}{z}} |f(a)|\int_{b_\theta}^\alpha \Phi(\f{s}{z}-a) s^l \diff s\diff a \dmz + \int_\theta^\eta\int_{\f{b_\theta}{z}}^{\frac{\alpha}{z}}|f(a)|\int_{za}^\alpha \Phi(\f{s}{z}-a) s^l \diff s \diff a \dmz \nonumber\\
		& \leq \int_\theta^{\eta}z^{l+1}\int_{b_\theta}^{\frac{b_\theta}{z}}\left(\Psi(\frac{b_\theta}{z}-a)-\Psi(\frac{\alpha}{z}-a)\right) |f(a)|a^l\diff a z\dmz + \int_\theta^{\eta}\int_{\f{b_\theta}{z}}^{\frac{\alpha}{z}}\left(1-\Psi(\frac{\alpha}{z}-a)\right) |f(a)|a^l\diff a z\dmz \nonumber\\
		& \leq \label{bordgauche} \theta^l\left(1-\Psi(\frac{\alpha-b_\theta}{\theta})\right),
\end{align}
since for $b_\theta < s \leq \frac{b_\theta}{z}$ we have $\frac{b_\theta}{z}-a\leq b$ and so $\Psi(\frac{b_\theta}{z}-a)=1.$ Taking $\alpha$ as closed to $b_\theta$ as needed, we obtain the first estimate of \eqref{massebord}.

As done before, we choose a $\beta$ to obtain a simpler expression of $T_\Sigma,$ namely $\beta > \eta\Sigma.$ Then, one has
\begin{align}
\int_\beta^\Sigma |T_\Sigma f(s)|s^l\diff s & \leq \int_\theta^\eta\int_{b_\theta}^\Sigma|f(a)|\int_{\beta}^{\Sigma}\Phi(\frac{s}{z}-a)s^l\diff s \diff a \dmz \nonumber\\
& \leq \int_\theta^\eta z^{l+1}\int_{b_\theta}^\Sigma|f(a)|a^l\int_{\frac{\beta}{z}-a}^{\frac{\Sigma}{z}-a}\Phi(\sigma)\diff \sigma \diff a \dmz \nonumber\\
&\leq \label{borddroit} \theta^l\sup_{|I|=\frac{\Sigma-\beta}{\theta}}\int_I \Phi(a)\diff a,
\end{align}
which is small when $\Sigma-\beta$ is small.

\medskip

We have checked the assumptions of Theorem~\ref{compactBrezis} for the family $\mathcal{F},$ so it is relatively compact.
\end{proof}

To prove the irreducibility of the operator $T_\Sigma,$ it is useful to notice that $T_\Sigma$ can be expressed differently after switching the two integrals. One has
\[T_\Sigma f(s)  = \left\{
\begin{array}{l}
\dis \int_{b_\theta}^{\f{s}{\eta}} f(a)\int_\theta^\eta\Phi(\f{s}{z}-a)\dmz \diff a + \int_{\f{s}{\eta}}^{\f{s}{\theta}} f(a)\int_\theta^{\f{s}{a}}\Phi(\f{s}{z}-a)\dmz \diff a,\qquad b_\theta \leq s < \theta \Sigma, \vspace{2mm}\\
\dis \int_{b_\theta}^{\f{s}{\eta}} f(a)\int_\theta^\eta\Phi(\f{s}{z}-a)\dmz \diff a + \int_{\f{s}{\eta}}^\Sigma f(a)\int_\theta^{\f{s}{a}}\Phi(\f{s}{z}-a)\dmz \diff a,\qquad \theta \Sigma \leq s \leq \eta \Sigma, \vspace{2mm}\\
\dis \int_{b_\theta}^\Sigma f(a)\int_\theta^\eta \Phi(\f{s}{z}-a)\dmz \diff a,\qquad \eta \Sigma < s \leq \Sigma. \end{array} \right.\]

\begin{lem} \label{irredTS}
Let $l$ be a nonpositive number. Under the hypotheses of Proposition~\ref{existencefSigma}, for all $\Sigma>\f{1}{1-\theta}b$, the operator $T_\Sigma : \Leb^1((b_\theta,\Sigma),s^l\diff s) \rightarrow \Leb^1((b_\theta,\Sigma),s^l\diff s)$ is irreducible.
\end{lem}

\begin{proof}
Let $J\neq \{0\}$ be a $T_\Sigma$-invariant ideal in $\Leb^1((b_\theta,\Sigma),s^l\diff s)$. There exists a subset $\omega \subset (b_\theta,\Sigma)$ such that $J=\{f \in \Leb^1((b_\theta,\Sigma),s^l\diff s)\ |\ \supp f \subset \omega \}.$ Let $f_\omega:=s^{-l}\mathds{1}_\omega$ and $s_0:=\inf \supp  f_\omega \geq b_\theta.$ Since $J\neq \{0\}$ (so $s_0<\Sigma)$ and $\theta<\eta,$  one can find $\zeta$ and $\xi$ both positive such that \[T_\Sigma f_\omega(s)\geq \int_{s_0}^{s_0+\zeta}\int_\theta^{\theta+\xi}\Phi(\frac{s}{z}-a)\dmz f_\omega(a)\diff a. \] 
For $s\geq s_0,$ the functions $z\mapsto \frac{s}{z}-a$ and $a\mapsto \frac{s}{z}-a$ are continuous decreasing functions. So, if $s$ is such that $\frac{s}{\theta}-s_0>b,$ then one  can choose $\zeta$ and $\xi$ such that for all $(s,a)\in [s_0,s_0+\zeta]\times [\theta, \theta + \xi],\ \frac{s}{z}-a \in \supp \Phi.$ Additionally, for each $\zeta>0,$ the integral $\int_{s_0}^{s_0+\zeta}f_\omega(a)\diff a$ is positive. We deduce that $[\theta(b+s_0),\Sigma]\subset \supp T_\Sigma f_\omega \subset [s_0,\Sigma],$ so $s_0 \leq \theta(b+s_0),$ which is equivalent to $s_0 \leq b_\theta.$ Finally $b_\theta = s_0,$ so $J=[b_\theta,\Sigma]$ and $T_\Sigma$ is irreducible.
\end{proof}
%$T_\Sigma : \Leb^1((b_\epsilon,\Sigma),s^l\diff s) \to \Leb^1((b_\epsilon,\Sigma),s^l\diff s)$ is a compact operator.
\begin{proof}[Proof of Proposition~\ref{existencefSigma}]
Lemma \ref{compactTSigma} shows that the set $\mathcal{F}$ is relatively compact in $\Leb^1((b_\theta,\Sigma),s^l\diff s)$, which is exactly saying that $T_\Sigma$ is a compact operator of $\Leb^1((b_\theta,\Sigma),s^l\diff s).$ With Lemma~\ref{irredTS} in addition, we can apply Theorem~\ref{thmKRdP} to the operator $T_\Sigma$ for $\Sigma>\frac{1}{1-\theta}b$ to obtain the existence of a nonnegative function $f_\Sigma \in  \Leb^1((b_\theta,\Sigma),s^l\diff s)$ which is an eigenvector of $T_\Sigma$ associated to the eigenvalue $\rho_\Sigma.$ Since this function is defined on a compact subset of $\R_+,$ it also belongs to $\Leb^1((b_\theta,\Sigma),(s^k+s^l)\diff s)$ for $k<k_0.$

\end{proof}

\subsection{Passing to the limit $\Sigma\to \infty$}

We now want to show that up to a subsequence, $(f_\Sigma)_\Sigma$ converges to a fixed point of $T$. To that end, in the rest of the article we extend the functions defined on $(b_\theta,\Sigma)$ to $\R_+$ by $0$ out of $(b_\theta,\Sigma)$. Then we obtain the following proposition
\begin{prop} \label{fixedpointT}
Under hypotheses~\eqref{massconservation}-\eqref{decayPsi} there exists a nonnegative and normalized fixed point $$f\in \Leb^1(\R_+,(s^k+s^l)\diff s)$$  for all $l\leq0$ and $k<k_0,$ of the operator $T.$ Additionally, $f$ is unique in$ \Leb^1(\R_+)$ and its support is $[b_\theta, \infty).$
\end{prop}
First, we will show that the sequence $(\rho_\Sigma)_\Sigma$ converges to $1$ as $\Sigma \to \infty$.
\begin{lem}
If $(\rho_\Sigma,f_\Sigma)$ is an eigenpair of the operator $T_\Sigma,$ then the following inequality holds true
\beq \label{inegaliterho} 1-\Psi\left((\frac{1}{\eta}-1)\Sigma\right) \leq \rho_\Sigma \leq 1-\Psi\left(\frac{\Sigma}{\theta}-b_\theta\right). \eeq
\end{lem}
\begin{proof}
Integrating the equality $\rho_\Sigma f_\Sigma=T_\Sigma f_\Sigma$ over $(b_\theta,\Sigma)$, one has
\begin{align*}
\rho_\Sigma\int_{b_\theta}^\Sigma f_\Sigma(s)\diff s  = & \int_{b_\theta}^{\theta \Sigma}\int_\theta^\eta \int_{b_\theta}^\frac{s}{z} \Phi(\f{s}{z}-a)f(a)\diff a \dmz \diff s =:(A)\\
	& \quad + \int_{\theta \Sigma}^{\eta \Sigma}\int_\theta^{\f{s}{\Sigma}}\int_{b_\theta}^\Sigma \Phi(\f{s}{z}-a)f(a)\diff a \dmz \diff s =:(B)\\
	& \qquad + \int_{\theta \Sigma}^{\eta \Sigma}\int_{\f{s}{\Sigma}}^\eta\int_{b_\theta}^\frac{s}{z} \Phi(\f{s}{z}-a)f(a)\diff a \dmz \diff s =:(C)\\
	& \qquad \quad + \int_{\eta \Sigma}^\Sigma \int_\theta^\eta \int_{b_\theta}^\Sigma \Phi(\f{s}{z}-a)f(a)\diff a \dmz \diff s =:(D)
\end{align*}

\begin{align*}
(A) & = \int_\theta^\eta\int_{b_\theta}^\frac{b_\theta}{z}f_\Sigma(a)\int_{\frac{b_\theta}{z}}^{\theta \Sigma} \Phi(\f{s}{z}-a)\diff s \diff a \dmz + \int_\theta^\eta\int_{\frac{b_\theta}{z}}^{\frac{\theta\Sigma}{z}}f_\Sigma(a)\int_{za}^{\theta \Sigma} \Phi(\f{s}{z}-a)\diff s \diff a \dmz \\
	& = \int_\theta^\eta z\int_{b_\theta}^\frac{b_\theta}{z}f_\Sigma(a)\left[\Psi(\frac{b_\theta}{z}-a)-\Psi(\frac{\theta \Sigma}{z}-a)\right] \diff a \dmz + \int_\theta^\eta z\int_{\frac{b_\theta}{z}}^{\frac{\theta\Sigma}{z}}f_\Sigma(a)\left[1-\Psi(\frac{\theta\Sigma}{z}-a)\right] \diff a \dmz \\
(B) & = \int_\theta^\eta f_\Sigma(a)\int_{b_\theta}^\Sigma\int_{z\Sigma}^{\eta \Sigma} \Phi(\f{s}{z}-a)\diff s \diff a \dmz \\
	& = \int_\theta^\eta z\int_{b_\theta}^\Sigma f_\Sigma(a)\left[\Psi(\Sigma-a)-\Psi(\frac{\eta \Sigma}{z}-a)\right] \diff a \dmz
	\end{align*}
	\begin{align*}
(C) & = \int_\theta^\eta\int_{b_\theta}^\frac{\theta \Sigma}{z}f_\Sigma(a)\int_{\theta\Sigma}^{z \Sigma} \Phi(\f{s}{z}-a)\diff s \diff a \dmz + \int_\theta^\eta\int_{\frac{\theta\Sigma}{z}}^\Sigma f_\Sigma(a)\int_{za}^{z \Sigma} \Phi(\f{s}{z}-a)\diff s \diff a \dmz \\
	& = \int_\theta^\eta z\int_{b_\theta}^\frac{\theta \Sigma}{z}f_\Sigma(a)\left[\Psi(\frac{\theta\Sigma}{z}-a)-\Psi(\Sigma-a)\right] \diff a \dmz + \int_\theta^\eta z\int_{\frac{\theta\Sigma}{z}}^\Sigma f_\Sigma(a)\left[1-\Psi(\Sigma-a)\right] \diff a \dmz \\
(D) & = \int_\theta^\eta\int_{b_\theta}^\Sigma f_\Sigma(a)\int_{\eta\Sigma}^{\Sigma} \Phi(\f{s}{z}-a)\diff s \diff a \dmz \\
	& = \int_\theta^\eta z\int_{b_\theta}^\Sigma f_\Sigma(a)\left[\Psi(\frac{\eta\Sigma}{z}-a)-\Psi(\frac{\Sigma}{z}-a)\right]\diff a \dmz
\end{align*}
Then notice that for $a\in(b_\theta,\Sigma)$ and $z\in (\theta,\eta)$ one has\[\frac{b_\theta}{z}-a \leq \frac{b_\theta}{z}-b_\theta = b_\theta\left(\frac{1}{z}-1\right)\leq b_\theta\left(\frac{1}{\theta}-1\right)= b,\] so as in the computations leading to \eqref{bordgauche}, $\Psi(\frac{b_\theta}{z}-a)=1.$ Combining these different expressions, we deduce \beq \label{egaliterho}\rho_\Sigma\int_{b_\theta}^\Sigma f_\Sigma(s)\diff s = \int_{b_\theta}^\Sigma f_\Sigma(s)\diff s - \int_\theta^\eta z\int_{b_\theta}^\Sigma \Psi(\frac{\Sigma}{z}-a)f_\Sigma(a) \diff a \dmz. \eeq
Using the fact that the function $\Psi$ is nonincreasing, we obtain the wanted inequality.
\end{proof}
Now we show that up to a subsequence, $(f_\Sigma)_\Sigma$ converges to a fixed point of $T$ denoted by $f$. Thanks to \eqref{inegaliterho} and the properties of $\Psi,$ we can define \[\Sigma_0:=\inf\left\{\Sigma>\max(\frac{1}{1-\theta}b,1)\mathrm{\ such\ that\ }\rho_\Sigma > \frac{1}{2}\right\}.\]

\begin{lem} \label{extraction}
Under hypotheses \eqref{massconservation}, \eqref{probameasure}, \eqref{supportB} and \eqref{decayPsi}, the set $\{f_\Sigma, \Sigma \geq \Sigma_0, \|f_\Sigma\|_{\Leb^1(\R_+,(s^k+s^l)\diff s)}=1\}$ has a compact closure in $\Leb^1(\R_+,(s^k+s^l)\diff s)$, for any $l\leq0$ and  $k<k_0,\ k_0$ being the real number given in~\eqref{decayPsi}.
\end{lem}

\begin{proof}
Let $k\leq0$ and $k\in[0,k_0).$
Once again, we apply Theorem~\ref{compactBrezis} to show the desired result. First, we show that \eqref{translation} hold true with $\Omega=(b_\theta, \infty)$ and $\mathcal{F}=\{f_\Sigma, \|f_\Sigma\|_{\Leb^1(\R_+,(s^k+s^l)\diff s)}=1\}.$ Let $\omega$ be a compact subset of $(b_\theta,\infty)$ and $b_\theta<\alpha<\beta$ such that $\omega \subset [\alpha,\beta].$ We use the following inequality
\begin{align*}
\|\tau_h f_\Sigma- f_\Sigma\|_{\Leb^1(\omega,(s^k+s^l)ds)} & \leq 2\|\tau_h T_\Sigma f_\Sigma-T_\Sigma f_\Sigma\|_{\Leb^1(\omega,(s^k+s^l)ds)} \\
	&\leq 2\left(\beta^k+\alpha^l\right)\|\tau_h T_\Sigma f_\Sigma-T_\Sigma f_\Sigma\|_{\Leb^1([\alpha,\beta])} \\
	&\leq 2\left(\beta^k+\alpha^l\right)\|\tau_h T_\Sigma f_\Sigma-T_\Sigma f_\Sigma\|_{\Leb^1([\alpha,\Sigma])}.
\end{align*}
The last quantity is small when $h$ is small uniformly with respect to $\Sigma$ since in the proof of Lemma~\ref{compactTSigma}, the estimates do not depend on the value of $\Sigma.$ To prove that \eqref{massebord} holds true, we use the estimate~\eqref{bordgauche} twice to write
\begin{align*}
\|f_\Sigma\|_{\Leb^1((b_\theta,\alpha),(s^k+s^l)ds)}&=\frac{1}{\rho_\Sigma}\int_{b_\theta}^\alpha T_\Sigma f_\Sigma(s)(s^k+s^l)ds \\
		&\leq 2\int_{b_\theta}^\alpha T_\Sigma f_\Sigma(s)s^lds + 2\alpha^k\int_{b_\theta}^\alpha T_\Sigma f_\Sigma(s)ds \\
		&\leq 2\theta^l\left(1-\Psi(\frac{\alpha-b_\theta}{\theta})\right)+2\alpha^k\left(1-\Psi(\frac{\alpha-b_\theta}{\theta})\right) \\
		& \leq 2(\theta^l+\alpha^k)\left(1-\Psi(\frac{\alpha-b_\theta}{\theta})\right)
\end{align*}
which is again independant of $\Sigma.$ The estimate~\eqref{borddroit} though depends on $\Sigma,$  so we write for $\Sigma$ larger than~$\beta$
\[
\rho_\Sigma\int_\beta^\Sigma f_\Sigma(a)\diff a = \int_\beta^{\theta\Sigma}T_\Sigma f_\Sigma(a)\diff a + \int_{\theta\Sigma}^{\eta \Sigma}T_\Sigma f_\Sigma(a)\diff a + \int_{\eta \Sigma}^\Sigma T_\Sigma f_\Sigma(a)\diff a.
\]
For the first integral, we compute
\begin{align*}
& \int_\beta^{\theta\Sigma}T_\Sigma f_\Sigma(a)\diff a \\
&=\int_\theta^\eta z\int_{b_\theta}^\frac{\beta}{z}\left[\Psi(\frac{\beta}{z}-a)-\Psi(\frac{\theta \Sigma}{z}-a)\right] f_\Sigma(a)\diff a \dmz + \int_\theta^\eta z\int_{\frac{\beta}{z}}^{\frac{\theta\Sigma}{z}}\left[1-\Psi(\frac{\theta\Sigma}{z}-a)\right] f_\Sigma(a)\diff a \dmz.
\end{align*}
The two other integrals correspond to the integrals $(B),\ (C)$ and $(D)$ from the previous proof. Combining the integrals, we obtain
\begin{align*}
&\rho_\Sigma\int_\beta^\Sigma f_\Sigma(a)\diff a \\
& = \int_\theta^\eta z\int_{b\theta}^\frac{\beta}{z}\Psi(\frac{\beta}{z}-a)f_\Sigma(a)\diff a  \dmz + \int_\theta^\eta z\int_\frac{\beta}{z}^\Sigma f_\Sigma(a)\diff a \dmz - \int_\theta^\eta z\int_{b\theta}^\Sigma \Psi(\frac{\Sigma}{z}-a)f_\Sigma(a)\diff a \dmz.
\end{align*}
We deal with the last integral using \eqref{egaliterho} and obtain after interverting integrals
\begin{align*}
& \rho_\Sigma\int_\beta^\Sigma f_\Sigma(a)\diff a = \int_{b_\theta}^\frac{\beta}{\eta}f_\Sigma(a)\int_\theta^\eta z\Psi(\frac{\beta}{z}-a) \dmz \diff a + \int_\frac{\beta}{\eta}^\frac{\beta}{\theta}f_\Sigma(a)\int_\theta^\frac{\beta}{a} z\Psi(\frac{\beta}{z}-a) \dmz \diff a \\
& \quad + \int_\frac{\beta}{\eta}^\frac{\beta}{\theta}f_\Sigma(a)\int_\frac{\beta}{a}^\eta z \dmz \diff a + \int_\frac{\beta}{\theta}^\Sigma f_\Sigma(a)\diff a  + \rho_\Sigma \int_{b_\theta}^\Sigma f_\Sigma(a)\diff a - \int_{b_\theta}^\Sigma f_\Sigma(a)\diff a \\
\iff & \int_{b_\theta}^\frac{\beta}{\eta} f_\Sigma(a)\diff a = \int_{b_\theta}^\frac{\beta}{\eta}f_\Sigma(a)\int_\theta^\eta z\Psi(\frac{\beta}{z}-a) \dmz \diff a + \int_\frac{\beta}{\eta}^\frac{\beta}{\theta}f_\Sigma(a)\int_\theta^\frac{\beta}{a} z\Psi(\frac{\beta}{z}-a) \dmz \diff a \\
& \quad + \rho_\Sigma \int_{b_\theta}^\beta f_\Sigma(a)\diff a - \int_\frac{\beta}{\eta}^\frac{\beta}{\theta}f_\Sigma(a)\int_\theta^\frac{\beta}{a}z \dmz \diff a \\
\iff & \int_\beta^\frac{\beta}{\eta} f_\Sigma(a) \int_\theta^\eta z\left[1- \Psi(\frac{\beta}{z}-a) \right] \dmz \diff a + (1-\rho_\Sigma)\int_{b_\theta}^\beta f_\Sigma(a)\diff a \\
& \quad + \int_\frac{\beta}{\eta}^\frac{\beta}{\theta}f_\Sigma(a)\int_\theta^\frac{\beta}{a}z\left[1-\Psi(\frac{\beta}{z}-a)\right] \dmz \diff a = \int_{b_\theta}^\beta f_\Sigma(a)\int_\theta^\eta z\Psi(\frac{\beta}{z}-a) \dmz \diff a
\end{align*}
Since $0<\theta < \eta < 1,$ we can choose $\beta > \frac{\eta}{1-\eta}b \geq b_\theta.$ In that case, $\frac{1}{\eta}-\frac{b}{\beta} > 1,$ and we can pick $r\in ]1,\frac{1}{\eta}-\frac{b}{\beta}[$ such that $\left(\frac{1}{\eta}-r\right)\beta > b.$ Noticing that $1-\Psi(\frac{\beta}{z}-a)$ and $1-\rho_\Sigma$ are nonnegative, we obtain
\[\int_\beta^{r\beta} f_\Sigma(a)\int_\theta^\eta z\left[1- \Psi(\frac{\beta}{z}-a) \right] \dmz \diff a \leq \int_{b_\theta}^\beta f_\Sigma(a)\int_\theta^\eta z\Psi(\frac{\beta}{z}-a) \dmz \diff a,\] then \[\left(1-\Psi((\frac{1}{\eta}-r)\beta)\right)\int_\beta^{r\beta}f_\Sigma(a)\diff a \leq \Psi((\frac{1}{\eta}-1)\beta),\] and finally \[\int_\beta^{r\beta}f_\Sigma(a)(a^l+a^k)\diff a \leq (\beta^l+(r\beta)^k)\frac{\Psi((\frac{1}{\eta}-1)\beta)}{\left(1-\Psi((\frac{1}{\eta}-r)\beta)\right)} \leq 4(r\beta)^k\Psi((\frac{1}{\eta}-1)\beta)\]
for $\beta$ large enough. We use this estimate to get
\begin{align*}
\int_\beta^\infty f_\Sigma(s)(s^k+s^l) \diff s & = \sum_{j=0}^{\infty}\int_{r^j \beta}^{r^{j+1} \beta} f_S(s) (s^k+s^l) \diff s \\
& \leq 4r^k\sum_{j=0}^{\infty}\left(r^j \beta\right)^k\Psi((\frac{1}{\eta}-1) r^j \beta) \\
& \leq 4Cr^k\sum_{j=0}^{\infty}\left(r^j \beta\right)^k\big((\frac{1}{\eta}-1) r^j \beta\big)^{-k_0} \\
& \leq \frac{C_{k, k_0,\eta,r}}{\beta^{k_0-k}}
\end{align*}
due to hypothesis \eqref{decayPsi}, for $\beta$ large enough.

\end{proof}
We are now ready to prove the existence and uniqueness of a fixed point for the operator $T.$
\begin{proof}[Proof of Proposition \ref{fixedpointT}]
We have proved in Lemma~\ref{extraction} that the set $\{f_\Sigma,\ \|f_\Sigma\|_{\Leb^1(\R_+,(s^l+s^k)\diff s)}\}$ has a compact closure in $\Leb^1(\R_+,(s^k+s^l)\diff s).$
We deduce the existence of $f\in \Leb^1(\R_+,(s^k+s^l)\diff s)$
such that, up to a subsequence still denoted by $(f_\Sigma)_\Sigma$, $f_\Sigma \to f$ strongly as $\Sigma \to+ \infty$.
Now we prove that the function $f$ is a fixed point of the operator $T.$
We use the following inequality
\[
\|f-Tf\|_{\Leb^1(\R_+,(s^k+s^l)\diff s)} \leq \|f-f_\Sigma\|_{\Leb^1(\R_+,(s^k+s^l)\diff s)} + (1-\rho_\Sigma) + \|T_\Sigma f_\Sigma-Tf\|_{\Leb^1(\R_+,(s^k+s^l)\diff s)}.
\]
The first term of the right-hand side tends to zero as $\Sigma$ tends to $\infty$ by definition of $f,$ and the second one is smaller than $\Psi\big( (\frac{1}{\eta}-1)\Sigma\big)$ according to \eqref{inegaliterho}. For the last one, we write
\begin{align*}
 & \|T_\Sigma f_\Sigma-Tf\|_{\Leb^1(\R_+,(s^k+s^l)\diff s)} \\
 & \leq \underbrace{\|T_\Sigma f_\Sigma - Tf_\Sigma\|_{\Leb^1(\R_+,(s^k+s^l)\diff s)}}_{=0} + \|T(f-f_\Sigma)\|_{\Leb^1(\R_+,(s^k+s^l)\diff s)} \\
 & \leq \|f-f_\Sigma\|_{\Leb^1(\R_+,(s^k+s^l)\diff s)}
\end{align*}
due to Lemma~\ref{lemmeTTSigma} and to the continuity of $T,$ which is proved in Lemma~\ref{stabiliteTL1poids}.

To prove uniqueness of the fixed point, we consider $f_1$ another nonnegative fixed point of $T$ in $\Leb^1(\R_+)$ satisfying $\int_0^\infty f_1(s)\diff s=\int_0^\infty f(s)\diff s.$ Recalling the definition \eqref{eq:defTfrag} of the operator $T$, the functions $f$ and $f_1$ satisfy the integral convolution equation
\beq \label{equationconvolutionfrag}
f(s)=\int_\theta^\eta \Phi*f(\frac{s}{z})\dmz.
\eeq
Since $f,\ f_1$ and $\Phi$ are in $\Leb^1(\R_+)$, their Laplace transforms exist on $\R_+$ and are continuous decreasing functions. Taking the Laplace transform of $f-f_1$ and switching integrals thanks to Fubini's theorem, one has for every $y\geq0$
\beq \label{eqconvtransflaplace}
\mathcal{L}[f-f_1](y) = \int_\theta^\eta\mathcal{L}[f-f_1](zy)\mathcal{L}[\Phi](zy)z\dmz.
\eeq
The Laplace transform $\mathcal{L}[f-f_1]$ is continuous on $\R_+$ and vanishes at the origin
\[\mathcal{L}[f-f_1](0)=\int_0^\infty f(s)\diff s - \int_0^\infty f_1(s)\diff s = 0.
\]
We now define the functions
\[
\overline{\mathrm{L}}(y)=\sup_{x\in [0,y]}\mathcal{L}[f-f_1](x)\quad \mathrm{\ and\ }\quad \underline{\mathrm{L}}(y)=\inf_{x\in [0,y]}\mathcal{L}[f-f_1](x).
\]
By continuity in $0$ of $\mathcal{L}[f-f_1]$ and because $\mathcal{L}[f-f_1](0)=0,$ one has
\[\forall y\geq 0,\qquad \overline{\mathrm{L}}(y)\geq 0,\ \underline{\mathrm{L}}(y)\leq 0.
\]
From \eqref{eqconvtransflaplace}, we obtain the inequality
\[\mathcal{L}[f-f_1](y)\leq \overline{\mathrm{L}}(\eta y)\int_\theta^\eta\mathcal{L}[\Phi](zy)z\dmz \leq \overline{\mathrm{L}}(\eta y),
\]
since $\Phi$ is a probability measure.
$\overline{L}$ is a continuous increasing function, so for all $x\leq y,$ one has
\[\mathcal{L}[f-f_1](x)\leq \overline{\mathrm{L}}(\eta x) \leq \overline{\mathrm{L}}(\eta y),
\]
from what we deduce
\beq \label{inegLbarre}
\overline{\mathrm{L}}(y)\leq \overline{\mathrm{L}}(\eta y).
\eeq
Iterating \eqref{inegLbarre}, we obtain for all $y\geq 0$ and all positive integer $j$
\[\overline{\mathrm{L}}(y)\leq \overline{\mathrm{L}}(\eta^j y).
\]
Letting $j\to\infty$ in this inequality and using the continuity of the function $\overline{\mathrm{L}}$ we obtain $\overline{\mathrm{L}}(y)=0$ for all nonnegative $y.$ With the same method, we show that $\underline{\mathrm{L}}(y)=0$ for all nonnegative $y,$ and finaly $\mathcal{L}[f-f_1]$ is the null function.
By the injectivity of the Laplace transform (Lerch's theorem \cite{Lerch}), one has $f=f_1.$

It remains to prove that $\supp f=[b_\theta,\infty).$ With the same kind of proof than the one we used for $T_\Sigma$, we can prove that $T$ is irreducible on $\Leb^1(b_\theta,\infty)$, and since $f$ is not the zero function we get the result.
\end{proof}

\subsection{Proof of the main theorem}

We are now ready to prove the main theorem of the paper. 
\begin{proof}[Proof of Theorem~\ref{mainthm}]
Combining Lemma~\ref{equivfixpt} and Proposition~\ref{fixedpointT}, we construct a solution to \eqref{equationMfrag}--\eqref{normalisationMfrag} using \[M(a,s):=\f{\psi(a)}{(a+s)^2}f(s).\] It remains to prove its uniqueness in the appropriate space. This solution belongs to $\Leb^1(\R_+^2, (1+s^2)\diff a \diff s)$ thanks to the following calculation
\begin{align*}
&\int_{b_\theta}^\infty\int_0^\infty M(a,s)(1+s^2)\diff a \diff s \\
& = \int_{b_\theta}^\infty \int_0^\infty \f{1}{(a+s)^2}f(s) \Psi(a) \diff a \diff s + \int_{b_\theta}^\infty \int_0^\infty \f{s^2}{(a+s)^2}f(s) \Psi(a) \diff a \diff s\\
& \leq \int_{b_\theta}^\infty f(s) s^{-2}\int_0^\infty \Psi(a) \diff a \diff s +  \int_{b_\theta}^\infty f(s) \int_0^\infty \Psi(a) \diff a \diff s \\
& = \|f\|_{\Leb^1((b_\theta,\infty),(1+s^{-2})\diff s)} < \infty
\end{align*}
because $f\in \Leb^1((b_\theta,\infty),s^l\diff s)$ for all nonpositive number $l.$
To prove the uniqueness of the solution $M\in\Leb^1(\R_+^2, (1+s^2)\diff a \diff s)$, consider another solution $M_1\in \Leb^1(\R_+^2, (1+s^2)\diff a \diff s)$. 
Necessarily, as for $M,$ there exists a measurable function $f_1$ such that for almost all $s\geq a\geq0$
\[M_1(a,s)=\f{\Psi(a)}{(a+s)^2}f_1(s).\]
For $0 < \alpha < \beta < \infty,$ we can write
\[\int_0^\infty f_1(s)\diff s=\frac{1}{\beta-\alpha}\int_0^\infty\int_\alpha^\beta \frac{(a+s)^2}{\Psi(a)}M_1(a,s)\diff a \diff s \leq \frac{2(\beta^2+1)}{(\beta-\alpha)\Psi(\beta)}\|M_1\|_{\Leb^1(\R_+,(1+s^2)\diff s)},\]
and this ensures that $f_1 \in \Leb^1(\R_+).$
Additionally we easily check as in Lemma~\ref{equivfixpt} that $f_1$ has to be a fixed point of $T.$
Then the uniqueness result in Proposition~\ref{fixedpointT} ensures that $f_1=f,$ and so $M_1=M.$
The existence and uniqueness of a solution to the initial problem~\eqref{equationNfrag}--\eqref{normalisationNfrag} follows from the relation~\eqref{relNM}.
\end{proof}

\section{Entropy and long time behaviour}\label{sec:entropy}

Now that we have solved the eigenvalue problem, we would like to characterize the asymptotic behaviour of a solution $n$ of \eqref{equationnfrag}--\eqref{bordnaissancenfrag}, as in \cite{PerthameTransport}. The General Relative Entropy principle provides informations about the evolution of the distance in $\Leb^1$ norm between a solution $n(t,\cdot,\cdot)$ and $\e^tN.$ To establish such useful inequalities, we use the formalism introduced in \cite{Michel2004} and \cite{Michel2005}. Strictly speaking, to use this method, we should prove some properties on a time-dependent solution $n,$ in particular its existence and uniqueness for any reasonable initial condition. Let us here assume the existence of such a solution, which moreover satisfies the common estimate (see \cite{PerthameTransport})
\beq \label{dominationN}
|n(t,a,x)|\leq Ce^tN(a,x), \qquad t,x > 0.
\eeq
It is usually ensured by the hypothesis $|n^0(a,x)|\leq CN(a,x)$ and a maximum principle. For $H$ a function defined on all $\R,$ we define, for $n\in \Leb^1(\R_+^2)$
\[\mathcal{H}[n]=\int_{b_\theta}^\infty\int_0^{x-b_\theta}xN(a,x)H\left(\frac{n(a,x)}{N(a,x)}\right)\diff a \diff x\]
which satisfies the following entropy property.
\begin{prop}
If $n$ is a solution of \eqref{equationnfrag}--\eqref{bordnaissancenfrag} satisfying \eqref{dominationN}, then
\beq \label{disspentropy}
\frac{d}{dt}\mathcal{H}[n(t,\cdot,\cdot)\e^{-t}]=-\mathcal{D}[n(t,\cdot,\cdot)\e^{-t}],
\eeq
with
\[\mathcal{D}[n]=\int_{b_\theta}^\infty x^2N(0,x)\left[\int_\theta^\eta \int_0^{\f{x}{z}-b_\theta}H\left(\f{n(a,\frac{x}{z}) }{N(a,\frac{x}{z})}\right)\diff \nu_x(a,z)-H\left(\int_\theta^\eta\int_0^{\frac{x}{z}-b_\theta}\f{n(a,\frac{x}{z})}{N(a,\frac{x}{z})}\diff \nu_x(a,z)\right)\right]\diff x
\]
where $\diff \nu_x(a,z)=\frac{B(a)N(a,\frac{x}{z})}{N(0,x)z^2}\diff a \dmz$ is a probability measure. Furthermore if $H$ is convex, then $\mathcal{D}\geq 0.$
%\beq
%\f{d}{dt}\iint x N H\left(\f{ne^{-t}}{N}\right)= \f{1}{8}\int_0^\infty x^2N\left(0,\f{x}{2}\right) \left[ H\left(\int_0^{x}\f{ne^{-t}}{N}\diff \mu_x(a)\right)-\int_0^{x-b} H\left(\f{ne^{-t}}{N}\right)\diff \mu_x(a)\right]\diff x
%\eeq
%for every function $H$ and with $\diff \mu_x(a) = \f{8B(a)N(a,x)}{N(0,x/2)}\diff a$ a probability mesure. In particular, if $H$ is any convex function, one has
%\[
%\f{d}{dt}\iint x N H\left(\f{ne^{-t}}{N}\right)\leq 0.
%\]
\end{prop}

Before proving this proposition, we make a remark about the conservative problem (\emph{i.e.} when only one daughter out of two is kept after division). In this case, the dominant eigenvalue is $0$ instead of $1$, and $xN(a,x)$ is an eigenvector associated with the eigenvalue $0,$ since the total mass is preserved. Then we obtain the equation
\beq \label{vpeqcons}\dda (x^2 N) + \ddx (x^2 N)= -x^2BN, \eeq
which might also be obtained multiplying \eqref{equationNfrag} by $x.$
\begin{proof}
Easy computations lead to \[ \ddt \f{n\e^{-t}}{N} + x\dda \f{n\e^{-t}}{N} + x\ddx \f{n\e^{-t}}{N}=0,\] where $N(a,x)>0$, \emph{i.e.} on the domain $\Omega:=\{x-a>b_\theta\}$. From this equality and \eqref{vpeqcons}, we deduce \beq \label{entropy:computations} \ddt \left(xNH \left(\f{n\e^{-t}}{N}\right) \right) + \dda \left(x^2NH \left(\f{n\e^{-t}}{N}\right) \right) + \ddx \left(x^2NH \left(\f{n\e^{-t}}{N}\right) \right)=-x^2BNH \left(\f{n\e^{-t}}{N}\right), \eeq and integrating \eqref{entropy:computations} over $\Omega,$ we obtain
\begin{align*}
& \f{d}{dt}\iint_{(b_\theta,\infty)\times (0,x-b_\theta)} x N H\left(\f{n\e^{-t}}{N}\right) \\
& = \int_{b_\theta}^\infty x^2 N(0,x)H\left(\f{n(t,0,x)\e^{-t}}{N(0,x)}\right)\diff x - \int_{b_\theta}^\infty x^2 N(x-b_\theta,x)H\left(\f{n(t,x-b_\theta,x)e^{-t}}{N(x-b_\theta,x)}\right)\diff x \\
& \quad + \int_0^\infty (a+b_\theta)^2 N(a,a+b_\theta)H\left(\f{n(t,a,a+b_\theta)\e^{-t}}{N(a,a+b_\theta)}\right)\diff a - \int_{b_\theta}^\infty \int_0^{x-b_\theta} x^2 B N H\left(\f{n\e^{-t}}{N}\right) \diff a \diff x \\
& = \int_{b_\theta}^\infty x^2N\left(0,x\right)H\left(\f{\e^{-t}}{N(0,x)}\int_\theta^\eta\int_0^{\frac{x}{z}-b_\theta}B(a)n(t,a,\frac{x}{z})\diff a \frac{\dmz}{z^2}\right)\diff x \\
& \quad - 2\int_\theta^\eta\int_{b_\theta}^\infty \int_0^{x-b_\theta} x^2 B N H\left(\f{n\e^{-t}}{N}\right) \diff a \diff x z\dmz\\
& = \int_{b_\theta}^\infty x^2N\left(0,x\right)H\left(\int_\theta^\eta\int_0^{\frac{x}{z}-b_\theta}\f{n(t,a,\frac{x}{z})\e^{-t}}{N(a,\frac{x}{z})}\diff \nu_x(a,z)\right)\diff x  \\
& \quad - 2\int_\theta^\eta\int_{zb_\theta}^{b_\theta} \int_0^{\f{x}{z}-b_\theta} x^2 B(a) N(a,\frac{x}{z}) H\left(\f{n(t,a,\frac{x}{z}) \e^{-t}}{N(a,\frac{x}{z})}\right) \diff a \diff x \f{\dmz}{z^2} \\
& \qquad - 2\int_\theta^\eta\int_{b_\theta}^\infty \int_0^{\f{x}{z}-b_\theta} x^2 B(a) N(a,\frac{x}{z}) H\left(\f{n(t,a,\frac{x}{z}) \e^{-t}}{N(a,\frac{x}{z})}\right) \diff a \diff x \f{\dmz}{z^2} \\
&=\int_{b_\theta}^\infty x^2N(0,x)\left[ H\left(\int_\theta^\eta\int_0^{\frac{x}{z}-b_\theta}\f{n(t,a,\frac{x}{z})\e^{-t}}{N(a,\frac{x}{z})}\diff \nu_x(a,z)\right)-\int_\theta^\eta \int_0^{\f{x}{z}-b_\theta}H\left(\f{n(t,a,\frac{x}{z}) \e^{-t}}{N(a,\frac{x}{z})}\right)\diff \nu_x(a,z)\right]\diff x,
\end{align*}
since for $x\in [zb_\theta,b_\theta]$ and $z\in [\theta,\eta], \f{x}{z}-b_\theta\leq b,$ and we conclude using Jensen's inequality.

\end{proof}
%Some particular cases are worth noticing. First, we can recover the uniqueness proved in Proposition \ref{fixedpointT}. Taking $n(t,a,x)=e^tN_1(a,x)$ with $N_1(a,x)$ another solution of the Perron problem \textcolor{red}{+domaine $\{b<x-a\}$, facile je pense} and $H$ any stricly convex function, we obtain that for all $x>b$, the function $a \mapsto N_1/N(a,x)$ is $\mu_x$-a.e. constant, so there exists a function $h$ such that $N_1/N(a,x)=h(x)$ $\mu_x$-a.e.. The function $N_1/N(a,x)$ being a solution of the transport equation derived from \eqref{transport}, it is also constant in $a$  $\mu_x$-a.e..
Appropriate choices of the function $H$ in \eqref{disspentropy} lead to interesting results. With $H(x)=x$, we recover the conservation law \eqref{conservationlaw}. Then taking $H(x)=|1-x|$, we obtain the decay of $\|N-ne^{-t}\|_{\Leb^1(\R_+, x\diff x \diff a)}$ as $t$ tends to infinity. In the case where the fragmentation kernel $\mu$ has a density with respect to the Lebesgue measure on $[0,1],$ we expect that this quantity will vanish, as in \cite{Michel2005,PerthameTransport}. In contrast, in the case of the equal mitosis, there is not hope for this distance to vanish. Indeed, one has an infinite number of eigentriplets $(\lb_j,N_j,\phi_j)$ with $j\in \Z$ defined by
\[
\lb_j=1+\f{2ij\pi}{\log 2}, \qquad N_j(a,x)=x^{1-\lb_j}N(a,x), \qquad \phi_j(a,x)=x^{\lb_j},
\]
so we expect a behavior as in \cite{BDG}, \emph{i.e.} the convergence of $n(t,a,x)\e^{-t}$ to the periodic solution \[\sum_{j\in \Z} \langle n^0,\phi_j\rangle\,\e^{\f{2ij\pi t}{\log 2}}N_j(a,x),\] where $\langle n,\phi\rangle=\int\int n(a,x)\phi(a,x)\diff a\diff x.$

%decay of the $\Leb^2(\R_+,\f{x}{N(a,x)}\mathds{1}_{N(a,x)>0}\diff a \diff x)$ distance between 
%For a convex function $H$, 
%\[\f{d}{dt}\iint x N H\left(\f{ne^{-t}}{N}\right) \leq 0
%\]

%In particular, if the solution $n$ we consider is $n(t,a,x)=e^tN_1(a,x)$ with $N_1$ another solution of \eqref{equationN}--\eqref{normalisationN},

%\begin{align*}
%\ddx ((x+y)\phi M) & = \phi \ddx ((x+y) M) + N(x+y)\ddx \phi \\
%& = \phi \left[-M-(x+y)B(x)M\right] + M\left[\phi + (x+y)B(x)\phi-2(x+y)B(x)\phi\left(0,\f{x+y}{2}\right)\right] \\
%& = -2(x+y)B(x)M\phi\left(0,\f{x+y}{2}\right)
%\end{align*}
%
%\begin{align*}
% (x+y)\ddx \f{N_1}{N} = & (x+y)\f{(x+y)N\ddx ((x+y)N_1)-(x+y)N_1\ddx ((x+y)N)}{\left((x+y)M\right)^2}\\
%= & \f{1}{N}\ddx ((x+y)N_1)-\f{N_1}{N^2}\ddx ((x+y)N) \\
%= & \f{1}{N}\left(-N_1-(x+y)B(x)N_1\right)-\f{N_1}{N^2}\left(-N-(x+y)B(x)N\right) \\
%= & 0
%\end{align*}

%\section{Oscillations}

%If $(1,N_1)$ is another eigenpair of problem \eqref{equationN}--\eqref{normalisationN}, then $n(t,a,x)=e^tN_1(a,x)$ is a solution of \eqref{equationn}--\eqref{bordincrementn}.

%Using relation $N(a,x)=M(a,x-a)$, one has an infinite number of eigentriplets $(\lb_j,N_j,\phi_j)$ with $j\in \Z$ defined by
%\beq
%\lb_j=1+\f{2ij\pi}{\log 2}, \qquad N_k(a,x)=x^{1-\lb_j}N(a,x), \qquad \phi_j(a,x)=x^{\lb_j}.
%\eeq

\section{Discussion and perspectives}\label{sec:conclusion}

We have proved the existence and uniqueness of a solution of the eigenproblem \eqref{equationNfrag}--\eqref{normalisationNfrag} in the special yet biologically relevant case of linear growth rate with a self-similar fragmentation kernel. Hypotheses on both this kernel and the division rate are fairly general.

As possible future work we can imagine to extend the result to general growth rates.
In this case the Perron eigenvalue is not explicit and it has to be determined in the same time as the eigenfunction, as in~\cite{Webb85,MetzDiekmann,Doumic2007}.
If we denote by $\lambda$ the eigenvalue, the equivalent of Equation \eqref{equationconvolutionfrag} is
\[P_\lb(s)=\int_0^1 e^{-\lb\int_s^{\frac sz}\frac{\diff u}{g(u)}}(\Phi*P_\lb) (\frac{s}{z})\frac{\dmz}{z}\]
with  $P_\lb(y)=e^{\lb \int_1^s\frac{\diff u}{g(u)}}M(0,s)$ and the equivalent of the solution given in \eqref{formethm} is \[N : (a,x) \mapsto \frac{\Psi(a)}{g(x)}e^{-\lambda\int_0^x \frac{\diff \alpha}{g(\alpha)}}P_\lambda(x-a).\]
Additionally for nonlinear growth rates, the function $(a,x)\mapsto x$ does not provide a conservation law as in~\eqref{conservationlaw}, and it has to be replaced by a solution to the dual Perron eigenproblem.
Such a dual eigenfunction appears in the definition of the General Relative Entropy \cite{Michel2004,Michel2005},
and for proving its existence one could follow the method in~\cite{PerthameRyzhik,DG} for the size-structured model.
Another possible generalization of the growth rate is adding variability, in the spirit of \cite{Rotenberg,MPR,Olivier2017}.
One might also consider a more general fragmentation kernel than in the case of self-similar fragmentation, or/and with a support which is not a compact subset of $(0,1).$

The other natural continuation of the present work is the proof of the well-posedness and the long-time behavior of the evolution equation, as in~\cite{Webb85,MetzDiekmann}.
To do so one can take advantage of the General Relative Entropy as in~\cite{Michel2005,CCM,BDG} or use general spectral methods~\cite{Webb87,MischlerScher}.

%Numerical aspects might also provide relevant continuations. For example, the numerical analysis of the power iteration used to approximate the fixed point $f$ would be a useful addition to our study. Inverse problems are also worth studying: estimating the function $B$ as in \cite{Doumic2009} and the fragmentation kernel $\mu$ seem both challenging and interesting.

% Indeed, now we have access to the asymptotic behaviour of the solution, a natural question is to quantify the accuracy of this model. First, one can perform an estimation of the division rate $B$ as it is done in  with the data in size and size-increment in the case of equal mitosis. To that end, one can use the fact that the division rate $B$ can be expression using only the function $\Phi$ introduced in \eqref{Phi}. Then, one can compare the incremental-size to other growth-fragmentation models, such as the age model or the size model, like in \cite{Robert2014}. An important consideration is that the steady distribution in both size and size-increment is difficult to obtain from experiments, so one can ask how to recover the division rate $B(a)$ or the fragmentation kernel $\mu$ only knowing the marginale in size, in the spirit of \cite{DET17}.

\

\section*{Acknowledgments}
The authors are very grateful to Marie Doumic for having suggested them the problem treated in this paper, and for the many fruitful discussions.

H.M. has been supported by the ERC Starting Grant SKIPPER$^{AD}$ (number 306321).
P.G. has been supported by the ANR project KIBORD, ANR-13-BS01-0004, funded by the French Ministry of Research.

%\bibliographystyle{abbrv}
%\bibliography{/home/humartin/Desktop/Articles/bibliojabref.bib}
%\bibliography{bibliojabref.bib}

\begin{thebibliography}{10}

\bibitem{BellAnderson}
G.~I. Bell and E.~C. Anderson.
\newblock Cell growth and division: I. a mathematical model with applications
  to cell volume distributions in mammalian suspension cultures.
\newblock {\em Biophysical Journal}, 7(4):329 -- 351, 1967.

\bibitem{BDG}
E.~Bernard, M.~Doumic, and P.~Gabriel.
\newblock Cyclic asymptotic behaviour of a population reproducing by fission
  into two equal parts.
\newblock Submitted, arXiv:1609.03846.

\bibitem{Brezis}
H.~Brezis.
\newblock {\em Functional analysis, {S}obolev spaces and partial differential
  equations}.
\newblock Universitext. Springer, New York, 2011.

\bibitem{CCM}
M.~J. C\'aceres, J.~A. Ca\~nizo, and S.~Mischler.
\newblock Rate of convergence to an asymptotic profile for the self-similar
  fragmentation and growth-fragmentation equations.
\newblock {\em J. Math. Pures Appl. (9)}, 96(4):334--362, 2011.

\bibitem{dePagter1986}
B.~de~Pagter.
\newblock Irreducible compact operators.
\newblock {\em {Math. Z.}}, 192(1):149--153, Mar 1986.

\bibitem{Doumic2007}
M.~Doumic.
\newblock {Analysis of a Population Model Structured by the Cells Molecular
  Content}.
\newblock {\em {Math. Model. Nat. Phenom.}}, 2(3):121--152, 2007.

\bibitem{DHKR}
M.~Doumic, M.~Hoffmann, N.~Krell, and L.~Robert.
\newblock Statistical estimation of a growth-fragmentation model observed on a
  genealogical tree.
\newblock {\em Bernoulli}, 21(3):1760--1799, 2015.

\bibitem{DG}
M.~Doumic~Jauffret and P.~Gabriel.
\newblock Eigenelements of a general aggregation-fragmentation model.
\newblock {\em Math. Models Methods Appl. Sci.}, 20(5):757--783, 2010.

\bibitem{Du}
Y.~Du.
\newblock {\em Order structure and topological methods in nonlinear partial
  differential equations. {V}ol. 1. Maximum principles and applications.}
\newblock Series in Partial Differential Equations and Applications. World
  Scientific Publishing Co. Pte. Ltd., Hackensack, NJ, 2006.
\newblock Maximum principles and applications.

\bibitem{Hall1991}
A.~J. Hall, G.~C. Wake, and P.~W. Gandar.
\newblock Steady size distributions for cells in one-dimensional plant tissues.
\newblock {\em {J. Math. Biol.}}, 30(2):101--123, Nov 1991.

\bibitem{Lerch}
M.~Lerch.
\newblock {Sur un point de la théorie des fonctions génératrices d'Abel}.
\newblock {\em {Acta Math.}}, 27:339--352, 1903.

\bibitem{MetzDiekmann}
J.~A.~J. Metz and O.~Diekmann, editors.
\newblock {\em The dynamics of physiologically structured populations},
  volume~68 of {\em Lecture Notes in Biomathematics}.
\newblock Springer-Verlag, Berlin, 1986.
\newblock %Papers from the colloquium held in Amsterdam, 1983.

\bibitem{Michel2004}
P.~Michel, S.~Mischler, and B.~Perthame.
\newblock General entropy equations for structured population models and
  scattering.
\newblock {\em {C. R., Math., Acad. Sci. Paris}}, 338(9):697 -- 702, 2004.

\bibitem{Michel2005}
P.~Michel, S.~Mischler, and B.~Perthame.
\newblock General relative entropy inequality: an illustration on growth
  models.
\newblock {\em {J. Math. Pures Appl. (9)}}, 84(9):1235 -- 1260, 2005.

\bibitem{MPR}
S.~Mischler, B.~Perthame, and L.~Ryzhik.
\newblock Stability in a nonlinear population maturation model.
\newblock {\em Math. Models Methods Appl. Sci.}, 12(12):1751--1772, 2002.

\bibitem{MischlerScher}
S.~Mischler and J.~Scher.
\newblock {Spectral analysis of semigroups and growth-fragmentation equations.}
\newblock {\em {Ann. Inst. Henri Poincar\'e, Anal. Non Lin\'eaire}},
  33(3):849--898, 2016.

\bibitem{Olivier2017}
A.~Olivier.
\newblock {How does variability in cells aging and growth rates influence the
  malthus parameter?}
\newblock {\em {Kinetic and Related Models}}, 10(2):481--512, June 2017.

\bibitem{PerthameTransport}
B.~Perthame.
\newblock {\em Transport equations in biology}.
\newblock Frontiers in Mathematics. Birkh\"auser Verlag, Basel, 2007.

\bibitem{PerthameRyzhik}
B.~Perthame and L.~Ryzhik.
\newblock Exponential decay for the fragmentation or cell-division equation.
\newblock {\em J. Differential Equations}, 210(1):155--177, 2005.

\bibitem{BurdenDouglas}
J.~D.~F. Richard L.~Burden.
\newblock {\em Numerical Analysis}.
\newblock Boston, MA: PWS Publishing Company; London: ITP International Thomson
  Publishing, 5th ed. edition, 1993.

\bibitem{Rotenberg}
M.~Rotenberg.
\newblock Transport theory for growing cell populations.
\newblock {\em J. Theoret. Biol.}, 103(2):181--199, 1983.

\bibitem{Sauls}
J.~T. Sauls, D.~Li, and S.~Jun.
\newblock Adder and a coarse-grained approach to cell size homeostasis in
  bacteria.
\newblock {\em Current Opinion in Cell Biology}, 38:38 -- 44, 2016.
\newblock Cell architecture.

\bibitem{Schaefer}
H.~H. Schaefer.
\newblock {\em Banach lattices and positive operators}.
\newblock Springer-Verlag, New York-Heidelberg, 1974.

\bibitem{SinkoStreifer}
J.~W. Sinko and W.~Streifer.
\newblock A new model for age-size structure of a population.
\newblock {\em Ecology}, 48(6):910--918, 1967.

\bibitem{T-A}
S.~Taheri-Araghi, S.~Bradde, J.~T. Sauls, N.~S. Hill, P.~A. Levin, J.~Paulsson,
  M.~Vergassola, and S.~Jun.
\newblock {Cell-Size Control and Homeostasis in Bacteria}.
\newblock {\em Current Biology}, 25(3):385 -- 391, 2015.

\bibitem{Webb85}
G.~F. Webb.
\newblock Dynamics of populations structured by internal variables.
\newblock {\em Math. Z.}, 189(3):319--335, 1985.

\bibitem{Webb}
G.~F. Webb.
\newblock {\em Theory of nonlinear age-dependent population dynamics},
  volume~89 of {\em Monographs and Textbooks in Pure and Applied Mathematics}.
\newblock Marcel Dekker, Inc., New York, 1985.

\bibitem{Webb87}
G.~F. Webb.
\newblock An operator-theoretic formulation of asynchronous exponential growth.
\newblock {\em Trans. Amer. Math. Soc.}, 303(2):751--763, 1987.

\end{thebibliography}

\end{document}